\documentclass[10pt]{article}
\usepackage{lmodern}
\usepackage{amsmath}
\usepackage[T1]{fontenc}
\usepackage[utf8]{inputenc}
\usepackage{authblk}
\usepackage{amsfonts}
\usepackage{graphicx}
\usepackage{rotating}
\usepackage{amssymb}
\usepackage[english]{babel}
\usepackage{xcolor}
\usepackage{amsthm}
\usepackage{graphicx}
\usepackage{mathrsfs}
\usepackage{makecell}
\usepackage{microtype}
\usepackage{mathscinet}
\usepackage{array}
\usepackage{multirow}
\usepackage{enumerate}
\usepackage[cal=boondoxo,bb=ams]{mathalfa}
\usepackage{hyperref}
\hypersetup{hidelinks}
\newtheorem{theorem}{Theorem}[section]
\newtheorem{prop}{Proposition}[section]
\newtheorem{lemma}{Lemma}[section]

\newtheorem{remark}{Remark}[section]

\newcommand{\ml}{\mathcal}
\newcommand{\mb}{\mathbb}
\DeclareMathOperator{\intt}{int}
\DeclareMathOperator{\extt}{ext}
\DeclareMathOperator{\bdd}{bdd}
\DeclareMathOperator{\lin}{lin}
\DeclareMathOperator{\nlin}{nlin}

\def\XXint#1#2#3{{\setbox0=\hbox{$#1{#2#3}{\int}$ }
		\vcenter{\hbox{$#2#3$ }}\kern-.6\wd0}}

\title{Large time asymptotic behavior for the dissipative Timoshenko system and its application}
\author[1]{Wenhui Chen\thanks{Wenhui Chen (wenhui.chen.math@gmail.com)}}
\affil[1]{School of Mathematics and Information Science, Guangzhou University,\authorcr 510006 Guangzhou, China}

\date{}

\begin{document}
		\maketitle

		\begin{abstract}
			\medskip
		In this paper, we study large time behavior for the dissipative Timoshenko system in the whole space $\mb{R}$, particularly, on the transversal displacement $w$ and the rotation angle $\psi$ of the filament for the beam. Different from decay properties of the energy term derived by Ide-Haramoto-Kawashima (2008) \cite{Ide-Haramoto-Kawashima=2008}, we discover new optimal growth $L^2$ estimates for the solutions themselves. Under the non-trivial mean condition on the initial data $w_1$, the unknowns $w$ and $\psi$ grow polynomially with the optimal rates $t^{3/4}$ and $t^{1/4}$, respectively, as large time. Furthermore, asymptotic profiles of them are introduced by the diffusion plate function, which explains a hidden cancellation mechanism in the shear stress $\partial_xw-\psi$. As an application of our results, we study the semilinear dissipative Timoshenko system with a power nonlinearity. Precisely, if the power is greater than the Fujita exponent, then the global in time existence of Sobolev solution is proved for the case of equal wave speeds, which partly gives a positive answer to the open problem in Racke-Said-Houari (2013) \cite{Racke-Said=2013}.
			\\
			
			
			\noindent\textbf{AMS Classification (2020)} 35L52, 35B40, 35L71, 35A01
		\end{abstract}
\fontsize{12}{15}
\selectfont

\section{Introduction}\label{Section_Introduction}\setcounter{equation}{0}
\subsection{Dissipative Timoshenko system}
 \hspace{5mm}In the 1920s, the innovative works \cite{Timoshenko-01,Timoshenko-02} from S.P. Timoshenko introduced a class of hyperbolic coupled systems for the vibration of the beam (called the Timoshenko beam), which may describe the transversal movement, the shear deformation and the rotational inertia effects. From a suitable scaling, the classical Timoshenko system may take the form
 \begin{align*}
\begin{cases}
\partial_t^2w-\partial_x(\partial_xw-\psi)=0,\\
\partial_t^2\psi-a^2\partial_x^2\psi-(\partial_xw-\psi)=0,
\end{cases}
 \end{align*}
with the time variable $t$ and the spatial variable $x$ denoting the point on the center line of the beam, where the unknowns $w=w(t,x)$ and $\psi=\psi(t,x)$ stand for the transversal displacement of the beam from the equilibrium state and the rotation angle of the filament of the beam, respectively. The positive constant $a$ is called the wave speed depending on some physical parameters of the beam (the density, the polar moment of inertia of a cross section, the Young modulus of elasticity, the moment of inertia of a cross section, and the shear modulus). Its physical background is referred to the monograph \cite{Graff=1975}.
 
  Due to the fact that some resistances cannot be neglected in real world problems, the so-called dissipative Timoshenko system was firstly considered in \cite{Munoz-Racke=2008}, which is addressed by
 \begin{align}\label{ALL-Dissipative-Timoshenko}
 	\begin{cases}
 		\partial_t^2w-\partial_x(\partial_xw-\psi)=0,\\
 		\partial_t^2\psi-a^2\partial_x^2\psi-(\partial_xw-\psi)+\partial_t\psi=0.
 	\end{cases}
 \end{align}
Its dissipation is produced here through the frictional damping $\partial_t\psi$ presenting in the equation of rotation angle only. The authors of \cite{Munoz-Racke=2008} considered this system \eqref{ALL-Dissipative-Timoshenko} in a bounded region with simple boundary conditions and demonstrated that the energy term decays exponentially when $a=1$ (the case of equal speeds) but polynomially when $a\neq 1$ (the case of non-equal speeds) as $t\to+\infty$. Then, qualitative properties of solutions to the dissipative Timoshenko system \eqref{ALL-Dissipative-Timoshenko} and its related models have drawn a lot of attention, including well-posedness, stability, large time behavior, global attractor, numerical analysis. We refer to \cite{Munoz-Racke=2008} and references given therein.

 Over the past twenty years, the corresponding Cauchy problem for the dissipative Timoshenko system, namely,
\begin{align}\label{Dissipative-Timoshenko}
\begin{cases}
\partial_t^2w-\partial_x(\partial_xw-\psi)=0,&x\in\mb{R},\ t>0,\\
\partial_t^2\psi-a^2\partial_x^2\psi-(\partial_xw-\psi)+\partial_t\psi=0,&x\in\mb{R},\ t>0,\\
(w,\partial_tw)(0,x)=(w_0,w_1)(x),&x\in\mb{R},\\
(\psi,\partial_t\psi)(0,x)=(\psi_0,\psi_1)(x),&x\in\mb{R},
\end{cases}
\end{align}
has been deeply studied. To begin with, via the energy term $U=U(t,x)$ such that
\begin{align}\label{Good-unknown-Kawashima}
U:=(\partial_xw-\psi,\partial_tw,a\partial_x\psi,\partial_t\psi)^{\mathrm{T}},
\end{align}
the authors of \cite{Ide-Haramoto-Kawashima=2008} employed energy methods in the phase space and the Fourier analysis to derive the regularity-loss type decay estimates
\begin{align}\label{New-01}
\|\partial_x^kU(t,\cdot)\|_{(L^2)^4}\lesssim (1+t)^{-\frac{1}{4}-\frac{k}{2}}\|U_0\|_{(L^1)^4}+\begin{cases}
\mathrm{e}^{-ct}\|\partial_x^kU_0\|_{(L^2)^4}&\mbox{if}\ \ a=1,\\
(1+t)^{-\frac{\ell}{2}}\|\partial_x^{k+\ell}U_0\|_{(L^2)^4}&\mbox{if}\ \ a\neq 1,
\end{cases}
\end{align}
with $k,\ell\in\mb{N}_0$, where $U_0=U_0(x)$ denotes the initial value $U(0,x)$. In other words, for the case of non-equal wave speeds $a\neq 1$ one has to assume the additional $\ell$-regularity on the Cauchy data to obtain polynomial decay estimates. Furthermore, by using asymptotic expansions of eigenvalues and their eigenprojections, some asymptotic profiles of the energy term $U$ were investigated. The study in the framework of weighted $L^1$ data was completed in \cite{Racke-Said=2013}, which leads to faster decay estimates. Subsequently, a refined energy method with modified Lyapunov functions was proposed by \cite{Mori-Kawashima=2014}, which allows that the dissipative energy term completely matches with the eigenvalues if $a\neq 1$. Employing the Carlson inequality associated with higher order energy methods, some $L^1$ estimates of $U(t,\cdot)$ with the weighted data were recently derived by \cite{Guesmia-Messaoudi=2023}. 

For corresponding nonlinear Cauchy problems to \eqref{Dissipative-Timoshenko}, we refer the interested reader to \cite{Ide-Kawashima=2008,Racke-Said=2013,Mori-Xu-Kawashima=2015,Xu-Mori-Kawashima=2015} and references therein. Among them, the paper \cite{Racke-Said=2013} studied the following semilinear dissipative Timoshenko system with the equal wave speeds:
\begin{align}\label{Nonlinear-Dissipative-Timoshenko}
	\begin{cases}
		\partial_t^2w^N-\partial_x(\partial_xw^N-\psi^N)=0,&x\in\mb{R},\ t>0,\\
		\partial_t^2\psi^N-\partial_x^2\psi^N-(\partial_xw^N-\psi^N)+\partial_t\psi^N=|\psi^N|^p,&x\in\mb{R},\ t>0,\\
		(w^N,\partial_tw^N)(0,x)=(w_0^N,w_1^N)(x),&x\in\mb{R},\\
		(\psi^N,\partial_t\psi^N)(0,x)=(\psi_0^N,\psi_1^N)(x),&x\in\mb{R},
	\end{cases}
\end{align}
with $p>1$, which is a corresponding nonlinear model to \eqref{Dissipative-Timoshenko} carrying the power nonlinearity with respect to the rotation angle. By weighted energy methods (introduced by \cite{Todorova-Yordanov=2001,Ikehata=2004,Ikehata-Inoue=2008}), the authors of \cite{Racke-Said=2013} demonstrated the global in time existence of small data solution when $p>12$ based on the corresponding energy term $U^N$, which is the same as \eqref{Good-unknown-Kawashima} with the superscript ``$N$''. To use some better decay estimates for this energy term, the additional assumption $U^N_0\in (L^{1,1})^4$ with $\int_{\mb{R}}U_0^N(x)\,\mathrm{d}x=0$ were taken. In \cite[Remark 7.2]{Racke-Said=2013}, they proposed an interesting open question.
\begin{center}
\emph{Does local in time solutions globally exist or blow up in the case $1<p\leqslant 12$?}
\end{center}
To understand this question and nominate a candidate for the critical exponent, formally taking $\partial_xw^{N}\equiv\psi^N$, the semilinear hyperbolic coupled system \eqref{Nonlinear-Dissipative-Timoshenko} is reduced to the semilinear classical damped wave equation
\begin{align}\label{Semilinear-Damped-Waves}
	\begin{cases}
		\partial_{t}^2\phi^N-\partial_x^2\phi^N+\partial_t\phi^N=|\phi^N|^p,&x\in\mb{R},\ t>0,\\
		(\phi^N,\partial_t\phi^N)(0,x)=(\phi^N_0,\phi^N_1)(x),&x\in\mb{R}.
	\end{cases}
\end{align}
It is well-known that the Fujita exponent 
\begin{align*}
	p_{\mathrm{Fuj}}(n):=1+\frac{2}{n}\ \ \mbox{with the dimension }n\in\mb{N}_+
\end{align*} 
discovered in \cite{Fujita=1966} for the semilinear heat equation, is the critical exponent for the semilinear damped wave equation \eqref{Semilinear-Damped-Waves} in $\mb{R}^n$ thanks to the diffusion phenomenon, where we refer the interested reader to \cite{Li-Zhou=1995,Todorova-Yordanov=2001,Zhang=2001,Ikeh-Tani-2005} and references therein for its detailed explanation. Hence, it is reasonable to conjecture that the critical exponent, namely, the threshold value $p=p_{\mathrm{Fuj}}(1)=3$ separates the blow-up range $1<p\leqslant p_{\mathrm{Fuj}}(1)$ from the small data solutions' global existence range $p>p_{\mathrm{Fuj}}(1)$ to the semilinear dissipative Timoshenko system \eqref{Nonlinear-Dissipative-Timoshenko}.  In the moment there is no answer to this question, that is, no result from the blow-up side and no result from the global in time existence of small data Sobolev solutions side if $1<p\leqslant 12$. As a byproduct of our methodology, we in this manuscript will give a positive answer of global in time existence if $p>p_{\mathrm{Fuj}}(1)=3$.

%For this reason, as considered by \cite[Remark 7.2]{Racke-Said=2013}, when $p>p_{\mathrm{Fuj}}(1)=3$ the local in time solution to the semilinear Cauchy problem \eqref{Nonlinear-Dissipative-Timoshenko} may be extended globally. However, this question is still open.

Concerning other studies on the Timoshenko system with various dissipations in the whole space $\mb{R}$, we refer to \cite{Said-Kasimov=2013,Mori-Kawashima=2014,Cao-Xu=2018,Wang-Xue=2019,Liu-Mao=2024} for the Fourier law of heat conduction (even with an additional frictional damping term), \cite{Racke-Said=2012,Said-Kasimov=2013,Mori-Kawashima=2016,Mori-Racke=2018} for the Cattaneo law of heat conduction, \cite{Said=2015} for the heat conduction of Green-Naghdi theory, \cite{Liu-Kawashima=2012,Khader-Said=2018,Mori=2018,Guesmia-Messaoudi=2023} for the memory-type dissipation, \cite{Soufyane-Said=2014,Racke-Wang-Xue=2017} for two frictional dissipations.

\subsection{Main purposes of this manuscript}
 \hspace{5mm}As we mentioned in the above, there are a lot of related works begun from the past twenty years in terms of the dissipative Timoshenko systems (with different kinds of dissipation) in the whole space $\mb{R}$. Their main focus is the energy terms related to \eqref{Good-unknown-Kawashima} by energy methods or spectral analysis. Consequently, a natural and important question arises.
\begin{center}
	\emph{Whether or not one can describe more detailed information of the transversal displacement $w$\\ and the rotation angle $\psi$ themselves for large time?}
\end{center}
To the best of knowledge of author, this question is open, even for the well-known dissipative Timoshenko system \eqref{Dissipative-Timoshenko}. The investigation of qualitative properties of the solutions $w,\psi$ themselves not only is significant for us to understand some underlying physical phenomena but also contribute to study global in time behavior for some corresponding nonlinear models (e.g. semilinear dissipative Timoshenko systems with $|\psi|^p$ in \cite{Racke-Said=2013,Racke-Wang-Xue=2017,Wang-Xue=2019} and our second purpose of this paper). In the present paper, we will partly give answers to the above question by deriving large time behavior for $w$ and $\psi$ to the dissipative Timoshenko system \eqref{Dissipative-Timoshenko}.

We stress that the study of $w$ and $\psi$ themselves is not simply a generalization of those for the energy terms. Recalling the pioneering paper \cite{Ide-Haramoto-Kawashima=2008} (see also \cite[Section 6]{Ueda-Duan-Kawashima=2012}), the standard energy term $U$ defined in \eqref{Good-unknown-Kawashima} decay polynomially due to the following pointwise estimate in the Fourier space:
\begin{align}\label{Point-wise-Kawashima}
|\widehat{U}|\lesssim\mathrm{e}^{-c\rho(|\xi|)t}|\widehat{U}_0|\ \ \mbox{with}\ \ \rho(|\xi|):=\begin{cases}
\displaystyle{\frac{|\xi|^2}{1+|\xi|^2}}&\mbox{if}\ \ a=1,\\[1em]
\displaystyle{\frac{|\xi|^2}{(1+|\xi|^2)^2}}&\mbox{if}\ \ a\neq1.
\end{cases}
\end{align}
Nevertheless, the solutions $w$ and $\psi$ themselves are not contained independently in the energy term $U$. It means that asymptotic behavior of $w,\psi$ is not a 
forthright consequence of \eqref{New-01}. For another, if one applies the sharp pointwise estimate \eqref{Point-wise-Kawashima} leading to
\begin{align}\label{New-02}
\chi_{\intt}(\xi)|\xi|\,|\widehat{\psi}|\lesssim\chi_{\intt}(\xi)\,\mathrm{e}^{-c|\xi|^2t}\left(|i\xi\widehat{w}_0-\widehat{\psi}_0|+|\widehat{w}_1|+a|\xi\widehat{\psi}_0|+|\widehat{\psi}_1|\right),
\end{align}
by directly dividing $|\xi|$ to estimate $\psi(t,\cdot)$ in the $L^2$ norm with the $L^1$ data (e.g. $\psi_1\in L^1$ or $\widehat{\psi}_1\in L^{\infty}$), the next unbounded term occurs:
\begin{align*}
\left\|\chi_{\intt}(\xi)|\xi|^{-1}\,\mathrm{e}^{-c|\xi|^2t}\right\|_{L^2}=C\left(\int_0^{\varepsilon_0}|\xi|^{-2}\,\mathrm{e}^{-2c|\xi|^2t}\,\mathrm{d}|\xi|\right)^{1/2}
\end{align*}
due to the strong singularity $|\xi|^{-2}$ as $|\xi|\to 0$.

Our first main goal in the present manuscript is to study large time asymptotic behavior of $w$ and $\psi$ to the dissipative Timoshenko system \eqref{Dissipative-Timoshenko}. To overcome some difficulties from the energy term $U$ (stated in the last paragraph), in Section \ref{Section_Fourier_Space} we reduce the hyperbolic coupled system \eqref{Dissipative-Timoshenko} to the fourth order hyperbolic scalar equations with respect to $w$ and $\psi$, independently, which recovers some oscillations to compensate strong singularities from \eqref{New-02}. Then, by using the WKB analysis and the Fourier analysis, in Section \ref{Section_L2-norm} we characterize the new optimal growth estimates
\begin{align*}
	\|w(t,\cdot)\|_{L^2}\approx t^{\frac{3}{4}} \ \ \mbox{and}\ \ \|\psi(t,\cdot)\|_{L^2}\approx t^{\frac{1}{4}}
\end{align*}
for large time $t\gg1$, provided that $\int_{\mb{R}}w_1(x)\,\mathrm{d}x\neq 0$. Furthermore, large time asymptotic profiles $w^{\mathrm{pf}}=w^{\mathrm{pf}}(t,x)$ and $\psi^{\mathrm{pf}}=\psi^{\mathrm{pf}}(t,x)$ of the diffusion plate type are introduced in \eqref{w-pf} and \eqref{psi-pf}, respectively, such that
\begin{align*}
	\lim\limits_{t\to+\infty}\|w(t,\cdot)-w^{\mathrm{pf}}(t,\cdot)\|_{L^2}=0 \ \
	\mbox{and}\ \ \lim\limits_{t\to+\infty}\|\psi(t,\cdot)-\psi^{\mathrm{pf}}(t,\cdot)\|_{L^2}=0.
\end{align*} 
As a byproduct of these profiles, we explain the construction of the good energy unknown $U$, especially, the hidden cancellation mechanism in the shearing stress $\partial_xw-\psi$ in Remark \ref{Rem-Not-Contradict}.

As an application of our results, the second main goal is to study the semilinear dissipative Timoshenko system \eqref{Nonlinear-Dissipative-Timoshenko}. Applying suitable $(L^2\cap L^1)-\dot{H}^k$ and $L^2-\dot{H}^k$ a priori estimates and the Banach fixed point argument (motivated by \cite{D-R=2014,Ebert-Reissig=2018}), in Section \ref{Section-GESDS} we demonstrate the global in time existence result of Sobolev solution when $p>p_{\mathrm{Fuj}}(1)$. It partly gives a positive answer to the open problem proposed in \cite[Remark 7.2]{Racke-Said=2013}. We also conjecture that by our philosophy the restriction on the power $p$ of the semilinear Timoshenko system with various dissipations in \cite{Racke-Wang-Xue=2017,Wang-Xue=2019} can be sharply improved.

\paragraph{\large Notation} The constants $c$ and $C$ may be changed from line to line but are independent of the time variable. We write $f\lesssim g$ if there exists a positive constant $C$ such that $f\leqslant Cg$, analogously for $f\gtrsim g$. The sharp relation $f\approx g$ holds if and only if $g\lesssim f\lesssim g$. We take the following zones of the Fourier space:
\begin{align*}
	\ml{Z}_{\intt}(\varepsilon_0):=\{|\xi|\leqslant\varepsilon_0\ll1\}, \ \ 
	\ml{Z}_{\bdd}(\varepsilon_0,N_0):=\{\varepsilon_0\leqslant |\xi|\leqslant N_0\},\ \ 
	\ml{Z}_{\extt}(N_0):=\{|\xi|\geqslant N_0\gg1\}.
\end{align*}
Moreover, the cut-off functions $\chi_{\intt}(\xi),\chi_{\bdd}(\xi),\chi_{\extt}(\xi)\in \mathcal{C}^{\infty}$ having their supports in the corresponding zones $\ml{Z}_{\intt}(\varepsilon_0)$, $\ml{Z}_{\bdd}(\varepsilon_0/2,2N_0)$ and $\ml{Z}_{\extt}(N_0)$, respectively, such that
\begin{align*}
	\chi_{\bdd}(\xi)=1-\chi_{\intt}(\xi)-\chi_{\extt}(\xi)\ \ \mbox{for all}\ \ \xi \in \mb{R}.
\end{align*}
The symbols of differential operators $|D|$ and $\langle D\rangle $ are denoted by $|\xi|$ and $\langle \xi\rangle$, respectively, where $\langle \xi\rangle^2:=1+|\xi|^2$ is  the Japanese bracket. Let us recall the $\gamma$-weighted $L^1$ space that is
\begin{align*}
	L^{1,\gamma}:=\left\{f\in L^1 \ \big|\ \|f\|_{L^{1,\gamma}}:=\int_{\mb{R}}(1+|x|)^\gamma|f(x)|\,\mathrm{d}x<+\infty \right\}\ \ \mbox{with}\ \ \gamma\in\mb{N}_0,
\end{align*}
so that $\|f\|_{L^1}\leqslant\|f\|_{L^{1,\gamma}}$. The means of a summable function $f$ are denoted by
\begin{align*}
	 P_f:=\int_{\mb{R}}f(x)\,\mathrm{d}x	\ \ \mbox{and}\ \  Q_f:=\int_{\mb{R}}xf(x)\,\mathrm{d}x.
\end{align*}
Let us introduce the initial data space for the sake of convenient
\begin{align*}
	\ml{D}_{\gamma}^s:=H^s\cap L^{1,\gamma} \ \ \mbox{with}\ \ s\in\mb{R},\ \gamma\in\mb{N}_0,
\end{align*}
which is the $s$-regular Sobolev space with the additional $\gamma$-weighted $L^1$ integrability, in particular, $\ml{D}_0^s=H^s\cap L^1$.

\section{Main results}\label{Section_Main_Results}\setcounter{equation}{0}
\subsection{Large time behavior for the linear dissipative Timoshenko system}
\hspace{5mm}Before stating our results on asymptotic behavior, let us introduce the auxiliary function via the Fourier transform as follows:
\begin{align*}
G(t,x):=\ml{F}^{-1}_{\xi\to x}\left(  \frac{\sin(c_a|\xi|^2t)}{c_a|\xi|^2}\,\mathrm{e}^{-\frac{1}{2}|\xi|^2t}\right)\ \ \mbox{carrying}\ \ c_a:=\frac{\sqrt{4a^2-1}}{2}>0,
\end{align*}
in which $G(t,x)$ is the so-called diffusion plate function since
\begin{itemize}
	\item its first component
	\begin{align*}
	\ml{K}_{\mathrm{dif}}(t,x):=\ml{F}^{-1}_{\xi\to x}\left(\mathrm{e}^{-\frac{1}{2}|\xi|^2t}\right)
	\end{align*}
is the kernel of the diffusion equation $\partial_tv-\frac{1}{2}\partial_x^2v=0$, and
	\item its second component
	\begin{align*}
	\ml{K}_{\mathrm{pla}}(t,x):=\ml{F}^{-1}_{\xi\to x}\left(\frac{\sin(c_a|\xi|^2t)}{c_a|\xi|^2}\right)
	\end{align*}
is the second kernel of the plate equation $\partial_t^2v+c_a^2\partial_x^4v=0$.
\end{itemize}
For these reasons, the large time property of $G(t,x)$ inherits partly the well-known decay property of Gaussian kernel $\ml{K}_{\mathrm{dif}}(t,x)$ and the growth property (see \cite[Theorem 1.1]{Ikehata=2024}) of plate kernel $\ml{K}_{\mathrm{pla}}(t,x)$. Then, let us take two functions to be the large time profiles (with the superscript ``$\mathrm{pf}$'')
\begin{align}
w^{\mathrm{pf}}(t,x)&:=G(t,x)P_{w_1}-\partial_x G(t,x) Q_{w_1}-\partial_xG(t,x)P_{\psi_0+\psi_1},\label{w-pf}\\
\psi^{\mathrm{pf}}(t,x)&:=\partial_xG(t,x)P_{w_1}.\label{psi-pf}
\end{align}
Notice that the factor $\partial_xG(t,x)$ includes a weaker singularity (as $|\xi|\to0$) than the one in $G(t,x)$, precisely,
\begin{align*}
\partial_xG(t,x)=i\,\ml{F}_{\xi\to x}^{-1}\left(\mathrm{sgn}(\xi)\frac{\sin(c_a|\xi|^2t)}{c_a|\xi|}\,\mathrm{e}^{-\frac{1}{2}|\xi|^2t} \right).
\end{align*}

For our first result on large time behavior for the transversal displacement $w$, its optimal growth rate $t^{3/4}$ is found. To get the large time convergence result, we next propose $w_1\in L^{1,2}$ to overcome the difficulty from the strong singularity $|\xi|^{-2}$ (if $|\xi|\to0$) in its leading term $G(t,x)P_{w_1}$.
\begin{theorem}\label{Thm-w}
Suppose that the initial data $(w_0,w_1)\in\ml{D}_0^0\times \ml{D}_0^{-1}$ and $(\psi_0,\psi_1)\in \ml{D}_0^{-1}\times\ml{D}_0^{-2}$. Then, the transversal displacement $w$ to the dissipative Timoshenko system \eqref{Dissipative-Timoshenko} with the wave speed $a>1/2$ satisfies the following optimal growth estimate:
	\begin{align*}
t^{\frac{3}{4}}|P_{w_1}|\lesssim 	\|w(t,\cdot)\|_{L^2}\lesssim t^{\frac{3}{4}}\|(w_0,w_1),(\psi_0,\psi_1)\|_{(\ml{D}_0^0\times\ml{D}_0^{-1})\times(\ml{D}_0^{-1}\times\ml{D}_0^{-2})}
	\end{align*}
for large time $t\gg1$, provided that $P_{w_1}\neq0$. Furthermore, assuming $w_1\in L^{1,2}$ as well as $\psi_0,\psi_1\in L^{1,1}$ additionally, and
\begin{align*}
	(w_0,w_1)\in\ml{D}_0^{\ell}\times\ml{D}_0^{\ell-1}\ \ \mbox{and}\ \ (\psi_0,\psi_1)\in\ml{D}_0^{\ell-1}\times\ml{D}_0^{\ell-2}\ \ \mbox{with}\ \ \ell>0\ \ \mbox{if}\ \ a\neq 1,
\end{align*}
then $w$ satisfies the following refined estimate:
\begin{align}\label{Error-w}
	\|w(t,\cdot)-w^{\mathrm{pf}}(t,\cdot)\|_{L^2}\lesssim\begin{cases}
		t^{-\frac{1}{4}}\|(w_0,w_1),(\psi_0,\psi_1)\|_{(\ml{D}_0^0\times\ml{D}_2^{-1})\times(\ml{D}_1^{-1}\times\ml{D}_1^{-2})}&\mbox{if}\ \ a=1,\\
		t^{-\min\left\{\frac{1}{4},\frac{\ell}{2}\right\}}\|(w_0,w_1),(\psi_0,\psi_1)\|_{(\ml{D}_0^{\ell}\times\ml{D}_2^{\ell-1})\times(\ml{D}_1^{\ell-1}\times\ml{D}_1^{\ell-2})}&\mbox{if}\ \ a\neq1,
	\end{cases}
\end{align}
for large time $t\gg1$.
\end{theorem}
\begin{remark}
If we simply consider $w^{\mathrm{sim}}(t,x):=G(t,x)P_{w_1}$ to be an asymptotic profile, then it is not difficult to prove $\|w(t,\cdot)-w^{\mathrm{sim}}(t,\cdot)\|_{L^2}\approx t^{1/4}$ as $t\gg1$ under a weaker assumption $w_1\in L^{1,1}$, which does not decay to zero as $t\to+\infty$. For our purpose of large time approximation, we construct the higher order profile $w^{\mathrm{pf}}(t,x)$ containing the initial data $\psi_0,\psi_1$.
\end{remark}

For the rotation angle $\psi$, its optimal (slow) growth rate $t^{1/4}$ is found. The weaker assumption $w_1\in L^{1,1}$ is proposed to get the large time convergence result thanks to the weaker singularity $|\xi|^{-1}$ (if $|\xi|\to 0$) in its leading term $\partial_xG(t,x)P_{w_1}$.
\begin{theorem}\label{Thm-psi}
	Suppose that the initial data $(w_0,w_1)\in\ml{D}_0^{-1}\times\ml{D}_0^{-2}$ and $(\psi_0,\psi_1)\in \ml{D}_0^0\times \ml{D}_0^{-1}$. Then, the rotation angle  $\psi$ to the dissipative Timoshenko system \eqref{Dissipative-Timoshenko} with the wave speed $a>1/2$ satisfies the following optimal growth estimate:
\begin{align*}
t^{\frac{1}{4}}|P_{w_1}|\lesssim	\|\psi(t,\cdot)\|_{L^2}\lesssim t^{\frac{1}{4}}\|(w_0,w_1),(\psi_0,\psi_1)\|_{(\ml{D}_0^{-1}\times\ml{D}_0^{-2})\times(\ml{D}_0^0\times \ml{D}_0^{-1})}
\end{align*}
for large time $t\gg1$, provided that $P_{w_1}\neq0$. Furthermore, assuming $w_1\in L^{1,1}$ additionally, and
\begin{align*}
(w_0,w_1)\in\ml{D}_0^{\ell-1}\times\ml{D}_0^{\ell-2}\ \ \mbox{and}\ \ (\psi_0,\psi_1)\in\ml{D}_0^{\ell}\times \ml{D}_0^{\ell-1}\ \ \mbox{with}\ \ \ell>0\ \ \mbox{if}\ \ a\neq 1,
\end{align*}
then $\psi$ satisfies the following refined estimate:
\begin{align*}
	\|\psi(t,\cdot)-\psi^{\mathrm{pf}}(t,\cdot)\|_{L^2}\lesssim \begin{cases}
		t^{-\frac{1}{4}}\|(w_0,w_1),(\psi_0,\psi_1)\|_{(\ml{D}_0^{-1}\times\ml{D}_1^{-2})\times(\ml{D}_0^0\times\ml{D}_0^{-1})}&\mbox{if}\ \ a=1,\\
		t^{-\min\left\{\frac{1}{4},\frac{\ell}{2} \right\}}\|(w_0,w_1),(\psi_0,\psi_1)\|_{(\ml{D}_0^{\ell-1}\times\ml{D}_1^{\ell-2})\times(\ml{D}_0^{\ell}\times\ml{D}_0^{\ell-1})}&\mbox{if}\ \ a\neq1,
	\end{cases}
\end{align*}
for large time $t\gg1$.
\end{theorem}

To end this subsection, we address some comments on both results.
\begin{remark}
The regularity-loss phenomenon does not occur in the estimates of $w,\psi$ due to the derived growth estimates (instead of decay estimates). Moreover, it is not surprising that $\psi$ grows slower than $w$ because of the frictional damping term $\partial_t\psi$ in the dissipative Timoshenko system \eqref{Dissipative-Timoshenko}. These results do not contradict to \cite[Theorem 2.1]{Ide-Haramoto-Kawashima=2008} and we will explain its detail in Remark \ref{Rem-Not-Contradict}.
\end{remark}

\begin{remark}
If one proposes $w_1\in L^{1,1}$ with $P_{w_1}=0$, in the case of equal wave speeds $a=1$ we may get a slower growth estimate for $w$ and a faster decay estimate for $\psi$, i.e. \eqref{Linear-01} and \eqref{Linear-02}, than those in Theorems \ref{Thm-w} and \ref{Thm-psi}, respectively. These estimates are used in the nonlinear problem \eqref{Nonlinear-Dissipative-Timoshenko}.
\end{remark}

\begin{remark}
In order to justify convergence results as $t\to+\infty$, i.e. some decay estimates for the error terms, we have to propose the regularity-loss effect with $\ell>0$ for the initial data in the case of non-equal wave speeds $a\neq 1$.
\end{remark}

\begin{remark}\label{Rem-Not-Contradict}
We contribute to explain the good unknown \eqref{Good-unknown-Kawashima} introduced in \cite[Formula (2.1)]{Ide-Haramoto-Kawashima=2008}, where actually there is a cancellation mechanism. Precisely, the leading term of $\partial_xw$ compensates the one of $\psi$, then $\partial_xw-\psi$ gets faster decay estimates than $\partial_xw$ and $\psi$. By using asymptotic expansions of $w$ and $\psi$ in the Fourier space (see Subsections \ref{Subsection-transversal displacement} and \ref{Subsection-Rotational-Angle}) we discover a cancellation stated in Remark \ref{Rem-Fourier-Cancellation}. From \eqref{oo-Err-01}, we derive
\begin{align*}
\left\|\chi_{\intt}(D)\big(\partial_xw(t,\cdot)-\psi(t,\cdot)\big)\right\|_{L^2}&\lesssim \left\|\chi_{\intt}(\xi)\,\mathrm{e}^{-c|\xi|^2t}\right\|_{L^2}\|(w_0,w_1),(\psi_0,\psi_1)\|_{(L^1\times L^1)\times (L^1\times L^1)}\\
&\lesssim (1+t)^{-\frac{1}{4}}\|(w_0,w_1),(\psi_0,\psi_1)\|_{(L^1\times L^1)\times (L^1\times L^1)},
\end{align*}
and, from Propositions \ref{Prop-w-large} as well as \ref{Prop-psi-large},
\begin{align*}
&\left\|\big(1-\chi_{\intt}(D)\big)\big(\partial_xw(t,\cdot)-\psi(t,\cdot)\big)\right\|_{L^2}\\
&\lesssim\begin{cases}
\mathrm{e}^{-ct}\|(w_0,w_1),(\psi_0,\psi_1)\|_{(H^1\times L^2)\times (L^2\times H^{-1})}&\mbox{if}\ \ a=1,\\
(1+t)^{-\frac{\ell}{2}}\|(w_0,w_1),(\psi_0,\psi_1)\|_{(H^{\ell+1}\times H^{\ell})\times (H^{\ell}\times H^{\ell-1})}&\mbox{if}\ \ a\neq 1.
\end{cases}
\end{align*}
Analogously, one can derive
\begin{align*}
&\|\partial_xw(t,\cdot)-\psi(t,\cdot)\|_{L^2}+\|\partial_tw(t,\cdot)\|_{L^2}+\|a\partial_x\psi(t,\cdot)\|_{L^2}+\|\partial_t\psi(t,\cdot)\|_{L^2}\\
&\lesssim
 (1+t)^{-\frac{1}{4}}\|(w_0,w_1),(\psi_0,\psi_1)\|_{(L^1\times L^1)\times (L^1\times L^1)}\\
 &\quad+\begin{cases}
	\mathrm{e}^{-ct}\|(w_0,w_1),(\psi_0,\psi_1)\|_{(H^1\times L^2)\times (H^1\times L^2)}&\mbox{if}\ \ a=1,\\
	(1+t)^{-\frac{\ell}{2}}\|(w_0,w_1),(\psi_0,\psi_1)\|_{(H^{\ell+1}\times H^{\ell})\times (H^{\ell+1}\times H^{\ell})}&\mbox{if}\ \ a\neq 1,
	\end{cases}
\end{align*}
which exactly coincides with the decay estimate of $U(t,\cdot)$ in the $L^2$ norm (see \cite[Theorem 2.1]{Ide-Haramoto-Kawashima=2008}). In conclusion, these theorems not only  do not contradict the classical results, but also discover a hidden cancellation mechanism.
\end{remark}

\begin{remark}
The philosophy of deriving optimal large time behavior of the transversal displacement and the rotation angle for the dissipative Timoshenko system \eqref{Dissipative-Timoshenko} in $\mb{R}$ is based on the reduction (to higher order evolution equations) and the refined Fourier analysis. Thus, we believe that our approach can be widely applied in a large class of dissipative Timoshenko systems with thermal dissipation \cite{Said-Kasimov=2013,Mori-Kawashima=2014}, memory-type dissipation \cite{Liu-Kawashima=2012}, or double dissipations \cite{Racke-Wang-Xue=2017}.
\end{remark}

\subsection{Global in time existence for the semilinear dissipative Timoshenko system}
\hspace{5mm}Let us state the global in time existence result in the super-critical case $p>p_{\mathrm{Fuj}}(1)=3$, which improves \cite[Theorem 6.4 with $p>12$]{Racke-Said=2013} and partly justifies the conjecture in \cite[Remark 7.2]{Racke-Said=2013}. It is an application of Theorems \ref{Thm-w} and \ref{Thm-psi} benefiting from some estimates for the solutions themselves (so we can avoid some losses in estimating the power nonlinearity).

\begin{theorem}\label{Thm-GESDS}
Suppose that the initial data $(w_0^N,w_1^N)\in\ml{D}_0^2\times \ml{D}_1^{1}$ and $(\psi_0^N,\psi_1^N)\in\ml{D}_0^{1}\times\ml{D}_0^{0}$ with $P_{w_1^N}=0$. Let
\begin{align*}
p>p_{\mathrm{Fuj}}(1)=3.
\end{align*}
Then, there exists $\epsilon>0$ such that for all $\|(w_0^N,w_1^N),(\psi_0^N,\psi_1^N)\|_{(\ml{D}_0^2\times \ml{D}_1^{1})\times(\ml{D}_0^{1}\times\ml{D}_0^{0})}\leqslant \epsilon$,
there is a uniquely determined Sobolev solution 
\begin{align*}
(w^N,\psi^N)\in\ml{C}([0,+\infty), H^2)\times\ml{C}([0,+\infty), H^{1})
\end{align*}
to the semilinear dissipative Timoshenko system \eqref{Nonlinear-Dissipative-Timoshenko} with the equal wave speeds. Furthermore, the following sharp estimates hold:
\begin{align*}
\|w^N(t,\cdot)\|_{\dot{H}^k}&\lesssim(1+t)^{\frac{1}{4}-\frac{k}{2}}\|(w_0^N,w_1^N),(\psi_0^N,\psi_1^N)\|_{(\ml{D}_0^2\times \ml{D}_1^{1})\times(\ml{D}_0^{1}\times\ml{D}_0^{0})}\ \ \,\ \mbox{for}\ \ k=0,2,\\
\|\psi^N(t,\cdot)\|_{\dot{H}^k}&\lesssim(1+t)^{-\frac{1}{4}-\frac{k}{2}}\|(w_0^N,w_1^N),(\psi_0^N,\psi_1^N)\|_{(\ml{D}_0^2\times \ml{D}_1^{1})\times(\ml{D}_0^{1}\times\ml{D}_0^{0})}\ \ \mbox{for}\ \ k=0,1.
\end{align*}
\end{theorem}
\begin{remark}
 Comparing with \cite[Theorem 6.4]{Racke-Said=2013} we do not propose the additional weighted integrability $g_0\in L^{1,1}$ with $P_{g_0}=0$ for $g_0=\mathrm{d}_x w_0^N-\psi_0^N,\mathrm{d}_x\psi_0^N,\psi_1^N$ due to our consideration of the solutions themselves instead of a total energy term. Although we just obtain a growth estimate for $\|w^N(t,\cdot)\|_{L^2}$, similarly to Remark \ref{Rem-Not-Contradict} we are able to recover the same decay estimates for an energy term as those in \cite[Theorem 6.4]{Racke-Said=2013}.
\end{remark}
\begin{remark}
As we will see later, the derived growth/decay rates in Theorem \ref{Thm-GESDS} exactly coincide with those of the corresponding linearized model in \eqref{Linear-01} and \eqref{Linear-02}, which verifies the effect of no loss of growth/decay.
\end{remark}

\begin{remark}
According to the connection between the semilinear dissipative Timoshenko system \eqref{Nonlinear-Dissipative-Timoshenko} and the semilinear classical damped wave equation \eqref{Semilinear-Damped-Waves}, we conjecture that the condition $p>p_{\mathrm{Fuj}}(1)$ of the global in time existence result in Theorem \ref{Thm-GESDS} is sharp, namely, in the sub-critical case $1<p<p_{\mathrm{Fuj}}(1)$ and the critical case $p=p_{\mathrm{Fuj}}(1)$ local in time solutions blow up in finite time.
\end{remark}

\section{Asymptotic behavior of solutions in the Fourier space}\label{Section_Fourier_Space}\setcounter{equation}{0}
\subsection{Reduction to fourth order differential models}\label{Subsection-Reduction}
\hspace{5mm}In order to study asymptotic behavior for $w$ and $\psi$, individually, different from \cite{Ide-Haramoto-Kawashima=2008,Racke-Said=2013,Mori-Kawashima=2014,Guesmia-Messaoudi=2023} by constructing the energy term $U$ as \eqref{Good-unknown-Kawashima}, our strategy is to derive asymptotic representations of them in the Fourier space via scalar fourth order differential equations. 

For one thing, by applying the damped wave operator $\partial_t^2-a^2\partial_x^2+\partial_t+I$ with the identity operator $I$ on \eqref{Dissipative-Timoshenko}$_1$ and using \eqref{Dissipative-Timoshenko}$_2$ to eliminate $\psi$ in the resultant equality, one derives the fourth order differential equation for the transversal displacement $w$ as follows:
\begin{align*}
0&=(\partial_t^2-a^2\partial_x^2+\partial_t+I)[(\partial_t^2-\partial_x^2)w+\partial_x\psi]\\
&=\partial_t^4w+\partial_t^3w+[1-(1+a^2)\partial_x^2]\partial_t^2w-\partial_t\partial_x^2w+a^2\partial_x^4w
\end{align*}
with its initial conditions $w_j=w_j(x)$ such that
\begin{align*}
\partial_t^2w(0,x)&=w''_0(x)-\psi'_0(x)=:w_2(x),\\ \partial_t^3w(0,x)&=w''_1(x)-\psi'_1(x)=:w_3(x).
\end{align*}

Analogously, by applying the free wave operator $\partial_t^2-\partial_x^2$ on \eqref{Dissipative-Timoshenko}$_2$ and using \eqref{Dissipative-Timoshenko}$_1$ to eliminate $w$ in the resultant equality, the rotation angle $\psi$ satisfies the following fourth order differential equation:
\begin{align*}
0&=(\partial_t^2-\partial_x^2)[(\partial_t^2-a^2\partial_x^2+\partial_t+I)\psi-\partial_xw]\\
&=\partial_t^4\psi+\partial_t^3\psi+[1-(1+a^2)\partial_x^2]\partial_t^2\psi-\partial_t\partial_x^2\psi+a^2\partial_x^4\psi 
\end{align*}
with its initial conditions $\psi_j=\psi_j(x)$ such that
\begin{align*}
	\partial_t^2\psi(0,x)&=a^2\psi_0''(x)-\psi_0(x)-\psi_1(x)+w'_0(x)=:\psi_2(x),\\
	\partial_t^3\psi(0,x)&=a^2\psi_1''(x)-a^2\psi''_0(x)+\psi_0(x)-w_0'(x)+w'_1(x)=:\psi_3(x).
\end{align*}

By setting $u=w,\psi$ for briefness, the unknown function $u=u(t,x)$ solves
\begin{align}\label{General-Fourth-order-eq}
\begin{cases}
\partial_t^4u+\partial_t^3u+[1-(1+a^2)\partial_x^2]\partial_t^2u-\partial_t\partial_x^2u+a^2\partial_x^4u=0,&x\in\mb{R},\ t>0,\\
(u,\partial_tu,\partial_t^2u,\partial_t^3u)(0,x)=(u_0,u_1,u_2,u_3)(x),&x\in\mb{R},
\end{cases}
\end{align}
with the Cauchy data $u_j=w_j,\psi_j$.
Although $w$ and $\psi$ satisfy the same partial differential equation equipped the hyperbolic (strictly hyperbolic if $a\neq1$) operator
\begin{align*}
\ml{L}_{\mathrm{Tim}}:=\partial_t^4-(1+a^2)\partial_t^2\partial_x^2+a^2\partial_x^4+\partial_t^3-\partial_t\partial_x^2+\partial_t^2,
\end{align*}
 qualitative properties of them may be quite different because of their initial conditions. Thanks to the well-known general well-posedness theory of hyperbolic equations in the whole space (see, for example, \cite[Subsection 3.4]{Ebert-Reissig=2018}), it is easy to conclude
 \begin{align}\label{Well-posedness}
 w,\psi\in \bigcap_{k=0,\dots,3}\ml{C}^k([0,T_0],H^{3-k})\ \ \mbox{for any}\ \ T_0>0,
 \end{align}
provided that $u_k\in H^{3-k}$ for $k=0,\dots,3$, that is $(w_0,w_1),(\psi_0,\psi_1)\in H^3\times H^2$. However, refined asymptotic behavior of $w$ or $\psi$ cannot be obtained by the general theory.

\subsection{Asymptotic expansions of characteristic roots}\label{Subsection-Charcteristic-Roots}
 \hspace{5mm}An application of the partial Fourier transform with respect to the spatial variable $x$ for the Cauchy problem \eqref{General-Fourth-order-eq} yields
\begin{align}\label{General-Fourth-order-eq-Fourier}
\begin{cases}
\mathrm{d}_t^4\widehat{u}+\mathrm{d}_t^3\widehat{u}+[1+(1+a^2)|\xi|^2]\mathrm{d}_t^2\widehat{u}+|\xi|^2\mathrm{d}_t\widehat{u}+a^2|\xi|^4\widehat{u}=0,&\xi\in\mb{R},\ t>0,\\
(\widehat{u},\mathrm{d}_t\widehat{u},\mathrm{d}_t^2\widehat{u},\mathrm{d}_t^3\widehat{u})(0,\xi)=(\widehat{u}_0,\widehat{u}_1,\widehat{u}_2,\widehat{u}_3)(\xi),&\xi\in\mb{R}, 
\end{cases}
\end{align}
whose characteristic equation is given by
\begin{align}\label{quartic-eq}
\lambda^4+\lambda^3+[1+(1+a^2)|\xi|^2]\lambda^2+|\xi|^2\lambda+a^2|\xi|^4=0.
\end{align}
Its discriminant is
\begin{align*}
0<\triangle_{\mathrm{Dis}}(|\xi|)=\begin{cases}
3(4a^2-1)|\xi|^4+O(|\xi|^6)&\mbox{if}\ \ \xi\in\ml{Z}_{\intt}(\varepsilon_0),\\
16a^2(a^2-1)^4|\xi|^{12}+O(|\xi|^{10})&\mbox{if}\ \ \xi\in\ml{Z}_{\extt}(N_0)\ \ \mbox{and} \ \ a\neq 1,\\
144|\xi|^8+O(|\xi|^6)&\mbox{if}\ \ \xi\in\ml{Z}_{\extt}(N_0)\ \ \mbox{and} \ \ a= 1,
\end{cases}
\end{align*}
thanks to our condition $a>1/2$, moreover,  $P_{\mathrm{Dis}}(|\xi|)=5+8(1+a^2)|\xi|^2>0$. We claim that for $\xi\in\ml{Z}_{\intt}(\varepsilon_0)\cup\ml{Z}_{\extt}(N_0)$ there are two pairs of non-real complex conjugate roots $\lambda_{1,2}$ and $\lambda_{3,4}$ to the quartic \eqref{quartic-eq}. For simplicity, they are denoted by
\begin{align}\label{Two-pairs-Conjugate-Roots}
\lambda_{1,2}=\lambda_{\mathrm{R}}^{(1)}+i\lambda_{\mathrm{I}}^{(1)}\ \ \mbox{and}\ \ \lambda_{3,4}=\lambda_{\mathrm{R}}^{(2)}+i\lambda_{\mathrm{I}}^{(2)},
\end{align}
where  $\lambda_{\mathrm{R}}^{(1)}, \lambda_{\mathrm{R}}^{(2)}, \lambda_{\mathrm{I}}^{(1)}, \lambda_{\mathrm{I}}^{(2)}\in\mb{R}\backslash\{0\}$ will be chosen later in different situations depending on the magnitude of $|\xi|\in(0,\varepsilon_0]\cup[N_0,+\infty)$ and the value of $a>1/2$.

It seems challenging to derive explicit roots to the quartic \eqref{quartic-eq}. Our purpose is to proceed asymptotic analysis for the $|\xi|$-dependent roots, which allows us to discuss behavior of the roots $\lambda_j=\lambda_j(|\xi|)$ with $j=1,\dots,4$ for different size of frequencies instead of their explicit forms.
\begin{itemize}
	\item \textbf{Small Frequencies}: Let us first consider $\xi\in\ml{Z}_{\intt}(\varepsilon_0)$ with $0<\varepsilon_0\ll 1$. We recall that the roots to \eqref{quartic-eq} are conjugate with two pairs and they can be expanded by
	\begin{align*}
	\lambda_j=\sum\limits_{k=0}^{+\infty}\lambda_{j,k}|\xi|^k\ \ \mbox{with}\ \ \lambda_{j,k}\in\mb{C}.
	\end{align*}
Plugging the last expansion into the quartic \eqref{quartic-eq} we are able to get
\begin{align*}
\lambda_{1,2}&=\left(-\frac{1}{2}\pm\frac{\sqrt{3}}{2}i\right)+\left(\frac{1}{2}\pm\frac{1+2a^2}{2\sqrt{3}}i\right)|\xi|^2+O(|\xi|^4),\\
\lambda_{3,4}&=\left(-\frac{1}{2}\pm\frac{\sqrt{4a^2-1}}{2}i\right)|\xi|^2+\left(a^2\mp\frac{a^2(a^2-1)}{\sqrt{4a^2-1}}i\right)|\xi|^4+O(|\xi|^6),
\end{align*}
for $\xi\in\ml{Z}_{\intt}(\varepsilon_0)$, in which we considered $4a^2-1>0$ due to $a>1/2$.
\begin{remark}
In the above expansions, we not only obtain pairwise distinct characteristic roots with negative real parts as those in \cite[Formulas (4.5)-(4.6)]{Ide-Haramoto-Kawashima=2008} but also derive some higher order terms, i.e. the $|\xi|^2$-terms in $\lambda_{1,2}$ and the $|\xi|^4$-terms in $\lambda_{3,4}$. These non-trivial higher order terms contribute to further expansions and asymptotic profiles of solutions.
\end{remark}
\item \textbf{Large Frequencies}: Let us turn to $\xi\in\ml{Z}_{\extt}(N_0)$ with $N_0\gg1$. We recall that the roots to \eqref{quartic-eq} are conjugate with two pairs and they can be expanded by
\begin{align*}
	\lambda_j=\bar{\lambda}_{j,1}|\xi|+\sum\limits_{k=0}^{+\infty}\bar{\lambda}_{j,-k}|\xi|^{-k}\ \ \mbox{with}\ \ \bar{\lambda}_{j,1},\bar{\lambda}_{j,-k}\in\mb{C}.
\end{align*}
Plugging the last expansion into the quartic \eqref{quartic-eq}, we need to distinguish it between two cases as follows: if $a=1$, then (it is a supplement of \cite{Ide-Haramoto-Kawashima=2008,Racke-Said=2013})
\begin{align*}
\lambda_{1,2}&=\pm i|\xi|+\frac{-1\pm \sqrt{3}i}{4}+O(|\xi|^{-1}),\\
\lambda_{3,4}&=\pm i|\xi|+\frac{-1\mp \sqrt{3}i}{4}+O(|\xi|^{-1});
\end{align*}
if $a\neq 1$, then
\begin{align*}
\lambda_{1,2}&=\pm i|\xi|+\frac{\pm i}{2(1-a^2)}|\xi|^{-1}-\frac{1}{2(1-a^2)^2}|\xi|^{-2}+O(|\xi|^{-3}),\\
\lambda_{3,4}&=\pm ia|\xi|-\frac{1}{2}+O(|\xi|^{-1}),
\end{align*}
for $\xi\in\ml{Z}_{\extt}(N_0)$. Note that we correct a typo, precisely, the coefficient of $|\xi|^{-2}$ when $a\neq 1$, in \cite[Formulas (4.9)-(4.10)]{Ide-Haramoto-Kawashima=2008} by an additional factor $1/2$.
\item \textbf{Bounded Frequencies}: Let us finally consider $\xi\in\ml{Z}_{\bdd}(\varepsilon_0,N_0)$ in which we determine the sign of $\mathrm{Re}\,\lambda_j$ for all $j=1,\dots,4$. We assume by contradiction that there exists $j_0\in\{1,\dots,4\}$ such that $\lambda_{j_0}=ib_{j_0}$ with $b_{j_0}\in\mb{R}\backslash\{0\}$, namely, a pure imaginary root to the quartic \eqref{quartic-eq}. It tells us that
\begin{align*}
	b_{j_0}^4-ib_{j_0}^3-[1+(1+a^2)|\xi|^2]b_{j_0}^2+i|\xi|^2b_{j_0}+a^2|\xi|^4=0,
\end{align*}
whose solution should fulfill
\begin{align*}
b_{j_0}^4-[1+(1+a^2)|\xi|^2]b_{j_0}^2+a^2|\xi|^4=0\ \ \mbox{as well as}\ \ b_{j_0}(b_{j_0}^2-|\xi|^2)=0.
\end{align*}
They yield a contradiction immediately due to $\xi\in\ml{Z}_{\bdd}(\varepsilon_0,N_0)$, and thus all roots cannot be pure imaginary. The continuity of characteristic roots with respect to $|\xi|$ associated with $\mathrm{Re}\,\lambda_j<0$ for $\xi\in\ml{Z}_{\intt}(\varepsilon_0)\cup\ml{Z}_{\extt}(N_0)$ shows that
\begin{align}\label{Negative-roots}
\mathrm{Re}\,\lambda_j<0\ \ \mbox{with}\ \ j=1,\dots,4,
\end{align}
for $\xi\in\ml{Z}_{\bdd}(\varepsilon_0,N_0)$.
\begin{remark}
The exponential stability was derived in \cite[Proposition 4.1]{Ide-Haramoto-Kawashima=2008} by energy methods in the Fourier space. In the above statement, thanks to the asymptotic expansions of characteristic roots when $\xi\in\ml{Z}_{\intt}(\varepsilon_0)\cup\ml{Z}_{\extt}(N_0)$ this stability also can be claimed by a contradiction argument.
\end{remark}
\end{itemize}

\subsection{Representation of solutions via conjugate roots in the Fourier space}
\hspace{5mm}This part contributes to determining the formal representation of $\widehat{u}$ to the Cauchy problem \eqref{General-Fourth-order-eq-Fourier} benefiting from two pairs of conjugate roots \eqref{Two-pairs-Conjugate-Roots} when $\xi\in\ml{Z}_{\intt}(\varepsilon_0)\cup\ml{Z}_{\extt}(N_0)$, that is
\begin{align*}
\widehat{u}&=d_1^{(1)}\mathrm{e}^{\lambda_1t}+d_1^{(2)}\mathrm{e}^{\lambda_2t}+d_2^{(1)}\mathrm{e}^{\lambda_3t}+d_2^{(2)}\mathrm{e}^{\lambda_4t}\\
&=\sum\limits_{k=1,2}\mathrm{e}^{\lambda_{\mathrm{R}}^{(k)}t}\left[(d_k^{(1)}+d_k^{(2)})\cos(\lambda_{\mathrm{I}}^{(k)}t)+i(d_k^{(1)}-d_k^{(2)})\sin(\lambda_{\mathrm{I}}^{(k)}t)\right],
\end{align*}
thanks to pairwise distinct characteristic roots \eqref{Two-pairs-Conjugate-Roots}. Here, the coefficients $d_k^{(1)}$ and $d_k^{(2)}$ of exponential functions $\mathrm{e}^{\lambda_jt}$ are determined according to
\begin{align}\label{linear-argebra}
	\underbrace{\left(
{\begin{array}{*{20}c}
			1 & 1 &1 &1\\
			\lambda_1 & \lambda_2 & \lambda_3 & \lambda_4\\
			\lambda_1^2 & \lambda_2^2 & \lambda_3^2 & \lambda_4^2\\
			\lambda_1^3 & \lambda_2^3 & \lambda_3^3 & \lambda_4^3\\
	\end{array}}
\right)}_{=:\mb{V}}
	\left(
{\begin{array}{*{20}c}
		d_1^{(1)} \\
		d_1^{(2)} \\
		d_2^{(1)} \\
		d_2^{(2)}\\
	\end{array}}
\right)
	=\underbrace{\left(
{\begin{array}{*{20}c}
			\widehat{u}_0 \\
			\widehat{u}_1 \\
			\widehat{u}_2 \\
			\widehat{u}_3 \\
	\end{array}}
\right)}_{=:\mb{D}}.
\end{align}
The determinant of this Vandermonde matrix $\mb{V}$ is represented as follows:
\begin{align*}
	\det(\mb{V})&=\prod\limits_{1\leqslant j<k\leqslant 4}(\lambda_k-\lambda_j)\\
	&=-4\lambda_{\mathrm{I}}^{(1)}\lambda_{\mathrm{I}}^{(2)}\left[(\lambda_{\mathrm{R}}^{(2)}-\lambda_{\mathrm{R}}^{(1)})^2+(\lambda_{\mathrm{I}}^{(2)}+\lambda_{\mathrm{I}}^{(1)})^2\right]\left[(\lambda_{\mathrm{R}}^{(2)}-\lambda_{\mathrm{R}}^{(1)})^2+(\lambda_{\mathrm{I}}^{(2)}-\lambda_{\mathrm{I}}^{(1)})^2\right].
\end{align*}
According to the well-known Cramer rule in the system of linear equations \eqref{linear-argebra}, the coefficients $d_{1,2}:=d_1^{(1,2)}$ and $d_{3,4}:=d_2^{(1,2)}$ are expressed by
\begin{align}\label{dj}
	d_j=\frac{\det(\mb{V}_j)}{\det(\mb{V})}\ \ \mbox{for}\ \ j=1,\dots,4,
\end{align}
in which $\mb{V}_j$ stands for the matrix formed by replacing the corresponding $j$-th column of $\mb{V}$ by the column vector $\mb{D}$ defined in \eqref{linear-argebra}. Then, carrying out lengthy but straightforward calculations, the determinants of $\mb{V}_j$ can be shown explicitly as follows:
\begin{align*}
	\det(\mb{V}_1)&=-2i\lambda_2\lambda_{\mathrm{I}}^{(2)}\left[(\lambda_{\mathrm{R}}^{(2)})^2+(\lambda_{\mathrm{I}}^{(2)})^2\right]\left[(\lambda_{\mathrm{R}}^{(2)}-\lambda_2)^2+(\lambda_{\mathrm{I}}^{(2)})^2\right]\widehat{u}_0\\
	&\quad+\left[2i\lambda_{\mathrm{I}}^{(2)}\left[(\lambda_{\mathrm{R}}^{(2)})^2+(\lambda_{\mathrm{I}}^{(2)})^2\right]^2+\lambda_2^2\lambda_4^2(\lambda_4-\lambda_2)-\lambda_2^2\lambda_3^2(\lambda_3-\lambda_2) \right]\widehat{u}_1\\
	&\quad+\left[-4i\lambda_{\mathrm{R}}^{(2)}\lambda_{\mathrm{I}}^{(2)}\left[(\lambda_{\mathrm{R}}^{(2)})^2+(\lambda_{\mathrm{I}}^{(2)})^2\right]-\lambda_2\lambda_4(\lambda_4^2-\lambda_2^2)+\lambda_2\lambda_3(\lambda_3^2-\lambda_2^2) \right]\widehat{u}_2\\
	&\quad+2i\lambda_{\mathrm{I}}^{(2)}\left[(\lambda_{\mathrm{R}}^{(2)}-\lambda_2)^2+(\lambda_{\mathrm{I}}^{(2)})^2\right]\widehat{u}_3,
\end{align*}
\begin{align*}
	\det(\mb{V}_2)&=2i\lambda_1\lambda_{\mathrm{I}}^{(2)}\left[(\lambda_{\mathrm{R}}^{(2)})^2+(\lambda_{\mathrm{I}}^{(2)})^2\right]\left[(\lambda_{\mathrm{R}}^{(2)}-\lambda_1)^2+(\lambda_{\mathrm{I}}^{(2)})^2\right]\widehat{u}_0\\
	&\quad-\left[2i\lambda_{\mathrm{I}}^{(2)}\left[(\lambda_{\mathrm{R}}^{(2)})^2+(\lambda_{\mathrm{I}}^{(2)})^2\right]^2+\lambda_1^2\lambda_4^2(\lambda_4-\lambda_1)-\lambda_1^2\lambda_3^2(\lambda_3-\lambda_1) \right]\widehat{u}_1\\
	&\quad-\left[-4i\lambda_{\mathrm{R}}^{(2)}\lambda_{\mathrm{I}}^{(2)}\left[(\lambda_{\mathrm{R}}^{(2)})^2+(\lambda_{\mathrm{I}}^{(2)})^2\right]-\lambda_1\lambda_4(\lambda_4^2-\lambda_1^2)+\lambda_1\lambda_3(\lambda_3^2-\lambda_1^2) \right]\widehat{u}_2\\
	&\quad-2i\lambda_{\mathrm{I}}^{(2)}\left[(\lambda_{\mathrm{R}}^{(2)}-\lambda_1)^2+(\lambda_{\mathrm{I}}^{(2)})^2\right]\widehat{u}_3,
\end{align*}
\begin{align*}
	\det(\mb{V}_3)&=-2i\lambda_4\lambda_{\mathrm{I}}^{(1)}\left[(\lambda_{\mathrm{R}}^{(1)})^2+(\lambda_{\mathrm{I}}^{(1)})^2\right]\left[(\lambda_{\mathrm{R}}^{(1)}-\lambda_4)^2+(\lambda_{\mathrm{I}}^{(1)})^2\right]\widehat{u}_0\\
	&\quad+\left[2i\lambda_{\mathrm{I}}^{(1)}\left[(\lambda_{\mathrm{R}}^{(1)})^2+(\lambda_{\mathrm{I}}^{(1)})^2\right]^2+\lambda_1^2\lambda_4^2(\lambda_4-\lambda_1)-\lambda_2^2\lambda_4^2(\lambda_4-\lambda_2) \right]\widehat{u}_1\\
	&\quad+\left[-4i\lambda_{\mathrm{R}}^{(1)}\lambda_{\mathrm{I}}^{(1)}\left[(\lambda_{\mathrm{R}}^{(1)})^2+(\lambda_{\mathrm{I}}^{(1)})^2\right]-\lambda_1\lambda_4(\lambda_4^2-\lambda_1^2)+\lambda_2\lambda_4(\lambda_4^2-\lambda_2^2) \right]\widehat{u}_2\\
	&\quad+2i\lambda_{\mathrm{I}}^{(1)}\left[(\lambda_{\mathrm{R}}^{(1)}-\lambda_4)^2+(\lambda_{\mathrm{I}}^{(1)})^2\right]\widehat{u}_3,
\end{align*}
\begin{align*}
	\det(\mb{V}_4)&=2i\lambda_3\lambda_{\mathrm{I}}^{(1)}\left[(\lambda_{\mathrm{R}}^{(1)})^2+(\lambda_{\mathrm{I}}^{(1)})^2\right]\left[(\lambda_{\mathrm{R}}^{(1)}-\lambda_3)^2+(\lambda_{\mathrm{I}}^{(1)})^2\right]\widehat{u}_0\\
	&\quad-\left[2i\lambda_{\mathrm{I}}^{(1)}\left[(\lambda_{\mathrm{R}}^{(1)})^2+(\lambda_{\mathrm{I}}^{(1)})^2\right]^2+\lambda_1^2\lambda_3^2(\lambda_3-\lambda_1)-\lambda_2^2\lambda_3^2(\lambda_3-\lambda_2) \right]\widehat{u}_1\\
	&\quad-\left[-4i\lambda_{\mathrm{R}}^{(1)}\lambda_{\mathrm{I}}^{(1)}\left[(\lambda_{\mathrm{R}}^{(1)})^2+(\lambda_{\mathrm{I}}^{(1)})^2\right]-\lambda_1\lambda_3(\lambda_3^2-\lambda_1^2)+\lambda_2\lambda_3(\lambda_3^2-\lambda_2^2) \right]\widehat{u}_2\\
	&\quad -2i\lambda_{\mathrm{I}}^{(1)}\left[(\lambda_{\mathrm{R}}^{(1)}-\lambda_3)^2+(\lambda_{\mathrm{I}}^{(1)})^2\right]\widehat{u}_3.
\end{align*}
Let us denote $\mb{V}_{1,2}=:\mb{V}_1^{(1,2)}$ and $\mb{V}_{3,4}=:\mb{V}_2^{(1,2)}$ for briefness. Summing up the last determinants of $\mb{V}_j$  and the coefficients $d_j$ defined in \eqref{dj}, we conclude the refined representation of solution 
\begin{align}
	\widehat{u}&=\sum\limits_{k=1,2}\underbrace{\mathrm{e}^{\lambda_{\mathrm{R}}^{(k)}t}\left[\cos(\lambda_{\mathrm{I}}^{(k)}t)\frac{\det(\mb{V}_k^{(1)})+\det(\mb{V}_k^{(2)})}{\det(\mb{V})}+i\sin(\lambda_{\mathrm{I}}^{(k)}t)\frac{\det(\mb{V}_k^{(1)})-\det(\mb{V}_k^{(2)})}{\det(\mb{V})}\right]}_{=:\widehat{u}^{(k)}}.\label{Rep-wide-u}
\end{align}
By making use of two pairs of conjugate characteristic roots, the representations of $\widehat{w}$ and $\widehat{\psi}$ for the dissipative Timoshenko system \eqref{Dissipative-Timoshenko} are found, where some cancellations in $\det(\mb{V}_k^{(1)})\pm\det(\mb{V}_k^{(2)})$ will be showed later.

\subsection{The transversal displacement in the Fourier space}\label{Subsection-transversal displacement}
\hspace{5mm}Recalling the transversal displacement $\widehat{w}$ in Subsection \ref{Subsection-Reduction} one knows
\begin{align*}
\widehat{w}_2=-|\xi|^2\widehat{w}_0-i\xi\widehat{\psi}_0\ \ \mbox{and}\ \ \widehat{w}_3=-|\xi|^2\widehat{w}_1-i\xi\widehat{\psi}_1.
\end{align*}
Let us apply the representation \eqref{Rep-wide-u} with $\widehat{u}=\widehat{w}$ to derive the next result for small frequencies.
\begin{prop}\label{Prop-w-small}
Let $\xi\in\ml{Z}_{\intt}(\varepsilon_0)$ with $0<\varepsilon_0\ll 1$. Then, $\widehat{w}$ fulfills the following pointwise estimates in the Fourier space:
\begin{align*}
\chi_{\intt}(\xi)|\widehat{w}|\lesssim \chi_{\intt}(\xi)\left(1+\frac{|\sin(c_a|\xi|^2t)|}{c_a|\xi|^2}\right)\mathrm{e}^{-c|\xi|^2t}\left(|\widehat{w}_0|+|\widehat{w}_1|+|\widehat{\psi}_0|+|\widehat{\psi}_1| \right)
\end{align*}
as well as
\begin{align}\label{Error-w02}
\chi_{\intt}(\xi)\left|\widehat{w}-\widehat{\ml{G}}(t,|\xi|)\left(\widehat{w}_1-i\xi(\widehat{\psi}_0+\widehat{\psi}_1)\right)\right|\lesssim \chi_{\intt}(\xi)\,\mathrm{e}^{-c|\xi|^2t}\left(|\widehat{w}_0|+|\widehat{w}_1|+|\widehat{\psi}_0|+|\widehat{\psi}_1| \right)
\end{align}
with the diffusion plate factor
\begin{align}\label{Diff-Plate-Factor}
\widehat{\ml{G}}(t,|\xi|):=\frac{\sin(c_a|\xi|^2t)}{c_a|\xi|^2}\,\mathrm{e}^{-\frac{1}{2}|\xi|^2t}\ \ \mbox{carrying}\ \ c_a=\frac{\sqrt{4a^2-1}}{2}>0.
\end{align}
\end{prop}
\begin{proof}
According to the asymptotic expansions of characteristic roots in Subsection \ref{Subsection-Charcteristic-Roots}, namely,
\begin{align*}
\lambda_{\mathrm{R}}^{(1)}&=-\frac{1}{2}+\frac{1}{2}|\xi|^2+O(|\xi|^4),\ \ \ \ \ \ \ \lambda_{\mathrm{I}}^{(1)}=\frac{\sqrt{3}}{2}+\frac{1+2a^2}{2\sqrt{3}}|\xi|^2+O(|\xi|^4),\\
\lambda_{\mathrm{R}}^{(2)}&=-\frac{1}{2}|\xi|^2+a^2|\xi|^4+O(|\xi|^6),\ \ \lambda_{\mathrm{I}}^{(2)}=c_a|\xi|^2-\frac{a^2(a^2-1)}{\sqrt{4a^2-1}}|\xi|^4+O(|\xi|^6),
\end{align*}
for $\xi\in\ml{Z}_{\intt}(\varepsilon_0)$ we can obtain
\begin{align}\label{det-V-small}
\det(\mb{V})=-\sqrt{3(4a^2-1)}|\xi|^2+O(|\xi|^4)
\end{align}
and
\begin{align*}
\chi_{\intt}(\xi)\left(|\det(\mb{V}_1^{(1)})|+|\det(\mb{V}_1^{(2)})| \right)\lesssim \chi_{\intt}(\xi)\left[|\xi|^4(|\widehat{w}_0|+|\widehat{w}_1|)+|\xi|^3(|\widehat{\psi}_0|+|\widehat{\psi}_1|)\right].
\end{align*}
Then, it leads to an exponential decay estimate
\begin{align}\label{Exponential-Decay-small}
\chi_{\intt}(\xi)|\widehat{w}^{(1)}|&\lesssim\chi_{\intt}(\xi)\,\mathrm{e}^{-ct}\left(|\cos(\lambda_{\mathrm{I}}^{(1)}t)|+|\sin(\lambda_{\mathrm{I}}^{(1)}t)|\right)\frac{|\det(\mb{V}_1^{(1)})|+|\det(\mb{V}_1^{(2)})|}{|\det(\mb{V})|}\notag\\
&\lesssim\chi_{\intt}(\xi)\,\mathrm{e}^{-ct}\left(|\widehat{w}_0|+|\widehat{w}_1|+|\widehat{\psi}_0|+|\widehat{\psi}_1|\right).
\end{align}

For this reason, the leading term of $\chi_{\intt}(\xi)\widehat{w}$ localizes in $\chi_{\intt}(\xi)\widehat{w}^{(2)}$ in the formula \eqref{Rep-wide-u}. 

Concerning the cosine part firstly, a direct calculation shows
\begin{align}\label{V+V}
\chi_{\intt}(\xi)|\det(\mb{V}_2^{(1)})+\det(\mb{V}_2^{(2)})|\lesssim\chi_{\intt}(\xi)|\xi|^2\left(|\widehat{w}_0|+|\widehat{w}_1|+|\widehat{\psi}_0|+|\widehat{\psi}_1| \right),
\end{align}
where the following $|\xi|^0$-terms in $\det(\mb{V}_2^{(1)})$ have been compensated by the corresponding $|\xi|^0$-terms in $\det(\mb{V}_2^{(2)})$ with the opposite signs:
\begin{align}\label{cancel}
2i\lambda_{\mathrm{I}}^{(1)}\left[(\lambda_{\mathrm{R}}^{(1)})^2+(\lambda_{\mathrm{I}}^{(1)})^2\right]^2\widehat{w}_1,\ \ -4i\lambda_{\mathrm{R}}^{(1)}\lambda_{\mathrm{I}}^{(1)}\left[(\lambda_{\mathrm{R}}^{(1)})^2+(\lambda_{\mathrm{I}}^{(1)})^2\right]\widehat{w}_2,\ \ 2i\lambda_{\mathrm{I}}^{(1)}\left[(\lambda_{\mathrm{R}}^{(1)})^2+(\lambda_{\mathrm{I}}^{(1)})^2\right]\widehat{w}_3.
\end{align}
As a consequence,
\begin{align}\label{w-cos-part}
\chi_{\intt}(\xi)\left|\mathrm{e}^{\lambda_{\mathrm{R}}^{(2)}t}\cos(\lambda_{\mathrm{I}}^{(2)}t)\frac{\det(\mb{V}_2^{(1)})+\det(\mb{V}_2^{(2)})}{\det(\mb{V})} \right|\lesssim\chi_{\intt}(\xi)\,\mathrm{e}^{-c|\xi|^2t}\left(|\widehat{w}_0|+|\widehat{w}_1|+|\widehat{\psi}_0|+|\widehat{\psi}_1|\right).
\end{align}

Let us next turn to the sine part. Unfortunately, the previous compensations do not hold for $\det(\mb{V}_2^{(1)})-\det(\mb{V}_2^{(2)})$ due to the subtraction. Instead, one observes
\begin{align*}
&\chi_{\intt}(\xi)\big|\left(\det(\mb{V}_2^{(1)})-\det(\mb{V}_2^{(2)})\right)-4i\lambda_{\mathrm{I}}^{(1)}\left[(\lambda_{\mathrm{R}}^{(1)})^2+(\lambda_{\mathrm{I}}^{(1)})^2\right]^2\widehat{w}_1\\
&\qquad\quad\,+8\lambda_{\mathrm{R}}^{(1)}\lambda_{\mathrm{I}}^{(1)}\left[(\lambda_{\mathrm{R}}^{(1)})^2+(\lambda_{\mathrm{I}}^{(1)})^2\right]\xi\widehat{\psi}_0-4\lambda_{\mathrm{I}}^{(1)}\left[(\lambda_{\mathrm{R}}^{(1)})^2+(\lambda_{\mathrm{I}}^{(1)})^2\right]\xi\widehat{\psi}_1\big|\\
&\lesssim\chi_{\intt}(\xi)|\xi|^2\left(|\widehat{w}_0|+|\widehat{w}_1|+|\widehat{\psi}_0|+|\widehat{\psi}_1| \right),
\end{align*}
and
\begin{align*}
\chi_{\intt}(\xi)\left|4i\lambda_{\mathrm{I}}^{(1)}\left[(\lambda_{\mathrm{R}}^{(1)})^2+(\lambda_{\mathrm{I}}^{(1)})^2\right]^2-2\sqrt{3}i\right|&\lesssim \chi_{\intt}(\xi)|\xi|^2,\\
\chi_{\intt}(\xi)\left|-8\lambda_{\mathrm{R}}^{(1)}\lambda_{\mathrm{I}}^{(1)}\left[(\lambda_{\mathrm{R}}^{(1)})^2+(\lambda_{\mathrm{I}}^{(1)})^2\right]\xi-2\sqrt{3}\,\xi \right|&\lesssim \chi_{\intt}(\xi)|\xi|^3,\\
\chi_{\intt}(\xi)\left|4\lambda_{\mathrm{I}}^{(1)}\left[(\lambda_{\mathrm{R}}^{(1)})^2+(\lambda_{\mathrm{I}}^{(1)})^2\right]\xi-2\sqrt{3}\,\xi  \right|&\lesssim \chi_{\intt}(\xi)|\xi|^3.
\end{align*}
Thus, an application of the triangle inequality yields
\begin{align*}
\chi_{\intt}(\xi)\left|\left(\det(\mb{V}_2^{(1)})-\det(\mb{V}_2^{(2)})\right)-2\sqrt{3}\left(i\widehat{w}_1+\xi\widehat{\psi}_0+\xi\widehat{\psi}_1\right) \right|\lesssim\chi_{\intt}(\xi)|\xi|^2\left(|\widehat{w}_0|+|\widehat{w}_1|+|\widehat{\psi}_0|+|\widehat{\psi}_1| \right).
\end{align*}
Furthermore, the error estimates hold
\begin{align*}
\chi_{\intt}(\xi)\left|\frac{1}{\det(\mb{V})}-\frac{1}{-\sqrt{3(4a^2-1)}|\xi|^2}\right|&\lesssim\chi_{\intt}(\xi),\\
\chi_{\intt}(\xi)\left|\sin(\lambda_{\mathrm{I}}^{(2)}t)-\sin(c_a|\xi|^2t)\right|&\lesssim\chi_{\intt}(\xi)|\xi|^4t,
\end{align*}
where we employed the mean value theorem in the last inequality, the boundedness of cosine function, and $|\lambda_{\mathrm{I}}^{(2)}-c_a|\xi|^2|=O(|\xi|^4)$. Summarizing the last three inequalities, by the chain triangle inequality we conclude
\begin{align}\label{w-sin-part}
&\chi_{\intt}(\xi)\left|i\sin(\lambda_{\mathrm{I}}^{(2)}t)\frac{\det(\mb{V}_2^{(1)})-\det(\mb{V}_2^{(2)})}{\det(\mb{V})}-\frac{\sin(c_a|\xi|^2t)}{c_a|\xi|^2}\left(\widehat{w}_1-i\xi(\widehat{\psi}_0+\widehat{\psi}_1)\right)  \right|\notag\\
&\lesssim\chi_{\intt}(\xi)(1+|\xi|^2t)\left(|\widehat{w}_0|+|\widehat{w}_1|+|\widehat{\psi}_0|+|\widehat{\psi}_1| \right).
\end{align}
The combination of \eqref{w-cos-part} and \eqref{w-sin-part} implies
\begin{align*}
\chi_{\intt}(\xi)\left|\widehat{w}^{(2)}-\frac{\sin(c_a|\xi|^2t)}{c_a|\xi|^2}\,\mathrm{e}^{-\frac{1}{2}|\xi|^2t}\left(\widehat{w}_1-i\xi(\widehat{\psi}_0+\widehat{\psi}_1)\right)\right|\lesssim\chi_{\intt}(\xi)\,\mathrm{e}^{-c|\xi|^2t}\left(|\widehat{w}_0|+|\widehat{w}_1|+|\widehat{\psi}_0|+|\widehat{\psi}_1| \right)
\end{align*}
and by the triangle inequality again the resultant leads to
\begin{align*}
\chi_{\intt}(\xi)|\widehat{w}^{(2)}|\lesssim \chi_{\intt}(\xi)\left(1+\frac{|\sin(c_a|\xi|^2t)|}{c_a|\xi|^2}\right)\mathrm{e}^{-c|\xi|^2t}\left(|\widehat{w}_0|+|\widehat{w}_1|+|\widehat{\psi}_0|+|\widehat{\psi}_1| \right),
\end{align*}
where we used the error estimate for the exponential function
\begin{align*}
\chi_{\intt}(\xi)\left|\mathrm{e}^{\lambda_{\mathrm{R}}^{(2)}t}-\mathrm{e}^{-\frac{1}{2}|\xi|^2t}\right|&=\chi_{\intt}(\xi)\,\mathrm{e}^{-\frac{1}{2}|\xi|^2t}\left|\big(a^2|\xi|^4+O(|\xi|^6)\big)\,t\int_0^1\mathrm{e}^{(a^2|\xi|^4+O(|\xi|^6))t\eta}\,\mathrm{d}\eta\right|\\
&\lesssim \chi_{\intt}(\xi)|\xi|^2\,\mathrm{e}^{-c|\xi|^2t}.
\end{align*}
Finally, recalling $\widehat{w}=\widehat{w}^{(1)}+\widehat{w}^{(2)}$ and the exponential decay estimate \eqref{Exponential-Decay-small}, our proof is complete.
\end{proof}

Let us next turn to some pointwise estimates of $\widehat{w}$ for large frequencies.
\begin{prop}\label{Prop-w-large}
Let $\xi\in\ml{Z}_{\extt}(N_0)$ with $N_0\gg 1$. Then, $\widehat{w}$ fulfills the following pointwise estimates in the Fourier space:
\begin{align*}
\chi_{\extt}(\xi)|\widehat{w}|\lesssim
\chi_{\extt}(\xi)\left(|\widehat{w}_0|+|\xi|^{-1}|\widehat{w}_1|+|\xi|^{-1}|\widehat{\psi}_0|+|\xi|^{-2}|\widehat{\psi}_1|\right)\times \begin{cases}
\mathrm{e}^{-ct}&\mbox{if}\ \ a=1,\\
\mathrm{e}^{-c|\xi|^{-2}t}&\mbox{if}\ \ a\neq1.
\end{cases}
\end{align*}
\end{prop}
\begin{proof}
We recall the asymptotic expansions of characteristic roots in Subsection \ref{Subsection-Charcteristic-Roots} that
\begin{align*}
	\lambda_{\mathrm{R}}^{(1)}&=-\frac{1}{2(1-a^2)^2}|\xi|^{-2}+O(|\xi|^{-3}),\ \ \lambda_{\mathrm{I}}^{(1)}=|\xi|+\frac{1}{2(1-a^2)}|\xi|^{-1}+O(|\xi|^{-3}),\\
	\lambda_{\mathrm{R}}^{(2)}&=-\frac{1}{2}+O(|\xi|^{-1}),\ \, \qquad\qquad\qquad
	\lambda_{\mathrm{I}}^{(2)}=a|\xi|+O(|\xi|^{-1}),
\end{align*}
in the case $a\neq 1$. Directly plugging the last expansions into the representations of $\det(\mb{V}_j^{(\ell)})$ for all $j,\ell\in\{1,2\}$, one finds
\begin{align*}
	\chi_{\extt}(\xi)\sum\limits_{j,\ell\in\{1,2\}}|\det(\mb{V}_j^{(\ell)})|\lesssim\chi_{\extt}(\xi)\sum\limits_{k=0,\dots,3}|\xi|^{6-k}|\widehat{w}_k|.
\end{align*}
Then, because of $\det(\mb{V})=-4a(1-a^2)^2|\xi|^6+O(|\xi|^5)$ and $\lambda_{\mathrm{R}}^{(1)}\sim -c|\xi|^{-2}$ for $\xi\in\ml{Z}_{\extt}(N_0)$, one may apply the bounded estimates $|\cos(\lambda_{\mathrm{I}}^{(k)}t)|, |\sin(\lambda_{\mathrm{I}}^{(k)}t)|\leqslant 1$ for $k=1,2$
to derive
\begin{align*}
	\chi_{\extt}(\xi)|\widehat{w}|&\lesssim \chi_{\extt}(\xi)\,\mathrm{e}^{-c|\xi|^{-2}t}\left(|\widehat{w}_0|+|\xi|^{-1}|\widehat{w}_1|+|\xi|^{-2}|\widehat{w}_2|+|\xi|^{-3}|\widehat{w}_3|\right)\\
	&\lesssim\chi_{\extt}(\xi)\,\mathrm{e}^{-c|\xi|^{-2}t}\left(|\widehat{w}_0|+|\xi|^{-1}|\widehat{w}_1|+|\xi|^{-1}|\widehat{\psi}_0|+|\xi|^{-2}|\widehat{\psi}_1|\right),
\end{align*}
which completes our proof for the general case $a\neq 1$. The case $a=1$ can be treated similarly.
\end{proof}

With analogous ways, concerning $k=0,\dots,3$, we may have
\begin{align*}
	\chi_{\extt}(\xi)|\mathrm{d}_t^k\widehat{w}|\lesssim
	\chi_{\extt}(\xi)|\xi|^{k}\left(|\widehat{w}_0|+|\xi|^{-1}|\widehat{w}_1|+|\xi|^{-1}|\widehat{\psi}_0|+|\xi|^{-2}|\widehat{\psi}_1|\right)\times \begin{cases}
		\mathrm{e}^{-ct}&\mbox{if}\ \ a=1,\\
		\mathrm{e}^{-c|\xi|^{-2}t}&\mbox{if}\ \ a\neq1,
	\end{cases}
\end{align*}
where the factor $|\xi|^k$ comes from the $k$ order time-derivative of $\sin(\lambda_{\mathrm{I}}^{(1,2)}t)$ and  $\cos(\lambda_{\mathrm{I}}^{(1,2)}t)$ with $\lambda_{\mathrm{I}}^{(1,2)}\sim c|\xi|$ for large frequencies. It is trivial from  the Plancherel theorem  that
\begin{align}\label{Est-well-posed-01}
\|\chi_{\extt}(D)\partial_t^kw(t,\cdot)\|_{H^{3-k}}&\lesssim\|\chi_{\extt}(\xi)|\xi|^{3-k}\mathrm{d}_t^k\widehat{w}(t,\xi)\|_{L^2}\notag\\
&\lesssim\left\|\big(\langle\xi\rangle^3|\widehat{w}_0|+\langle\xi\rangle^2|\widehat{w}_1|+\langle\xi\rangle^2|\widehat{\psi}_0|+\langle\xi\rangle|\widehat{\psi}_1| \big)\right\|_{L^2}\notag\\
&\lesssim\|(w_0,w_1),(\psi_0,\psi_1)\|_{(H^3\times H^2)\times(H^2\times H^1)}.
\end{align}
Due to the fact that the well-posed result of linear Cauchy problem is determined by the part for large frequencies,  the last estimate verifies the well-posedness for $w$ in \eqref{Well-posedness} provided that $(w_0,w_1)\in H^3\times H^2$ and $(\psi_0,\psi_1)\in H^2\times H^1$.

\subsection{The rotation angle in the Fourier space}\label{Subsection-Rotational-Angle}
\hspace{5mm}Recalling the rotation angle $\widehat{\psi}$ in Subsection \ref{Subsection-Reduction} one knows
\begin{align*}
	\widehat{\psi}_2=-(1+a^2|\xi|^2)\widehat{\psi}_0-\widehat{\psi}_1+i\xi\widehat{w}_0\ \ \mbox{and}\ \ \widehat{\psi}_3=(1+a^2|\xi|^2)\widehat{\psi}_0-a^2|\xi|^2\widehat{\psi}_1-i\xi\widehat{w}_0+i\xi\widehat{w}_1.
\end{align*}
Let us apply the representation \eqref{Rep-wide-u} with $\widehat{u}=\widehat{\psi}$ to derive the next result for small frequencies, whose proof is similar to the one of Proposition \ref{Prop-w-small} so we omit some details.
\begin{prop}\label{Prop-psi-small}
	Let $\xi\in\ml{Z}_{\intt}(\varepsilon_0)$ with $0<\varepsilon_0\ll 1$. Then, $\widehat{\psi}$ fulfills the following pointwise estimates in the Fourier space:
	\begin{align*}
		\chi_{\intt}(\xi)|\widehat{\psi}|\lesssim \chi_{\intt}(\xi)\left(1+\frac{|\sin(c_a|\xi|^2t)|}{c_a|\xi|}\right)\mathrm{e}^{-c|\xi|^2t}\left(|\widehat{w}_0|+|\widehat{w}_1|+|\widehat{\psi}_0|+|\widehat{\psi}_1| \right)
	\end{align*}
	as well as
	\begin{align}\label{Error-psi02}
		\chi_{\intt}(\xi)|\widehat{\psi}-i\xi\widehat{\ml{G}}(t,|\xi|)\widehat{w}_1|\lesssim\chi_{\intt}(\xi)\,\mathrm{e}^{-c|\xi|^2t}\left(|\widehat{w}_0|+|\widehat{w}_1|+|\widehat{\psi}_0|+|\widehat{\psi}_1| \right), 
	\end{align}
where the diffusion plate factor $\widehat{\ml{G}}(t,|\xi|)$ is showed in \eqref{Diff-Plate-Factor}.
\end{prop}
\begin{proof}
Reviewing the asymptotic expansions of $\lambda_{\mathrm{R}}^{(1,2)}$ and $\lambda_{\mathrm{I}}^{(1,2)}$, one may derive
\begin{align*}
\chi_{\intt}(\xi)\left(|\det(\mb{V}_1^{(1)})|+|\det(\mb{V}_1^{(2)})|\right)\lesssim\chi_{\intt}(\xi)\left[|\xi|^3(|\widehat{w}_0|+|\widehat{w}_1|)+|\xi|^2(|\widehat{\psi}_0|+|\widehat{\psi}_1|) \right]
\end{align*}
and, from \eqref{det-V-small},
\begin{align*}
\chi_{\intt}(\xi)|\widehat{\psi}^{(1)}|\lesssim\chi_{\intt}(\xi)\,\mathrm{e}^{-ct}\left(|\widehat{w}_0|+|\widehat{w}_1|+|\widehat{\psi}_0|+|\widehat{\psi}_1|\right).
\end{align*}

We are now turning to the estimate for $\chi_{\intt}(\xi)\widehat{\psi}^{(2)}$ where one may look back \eqref{Rep-wide-u}. In regard to the cosine part, by the same cancellations of $|\xi|^0$-terms as \eqref{cancel}, we notice that \eqref{V+V} as well as \eqref{w-cos-part} still hold for the case with $\widehat{\psi}$ instead of $\widehat{w}$. For the sine part, one may estimate
\begin{align*}
	&\chi_{\intt}(\xi)\big|\left(\det(\mb{V}_2^{(1)})-\det(\mb{V}_2^{(2)})\right)-4i\lambda_{\mathrm{I}}^{(1)}\left[(\lambda_{\mathrm{R}}^{(1)})^2+(\lambda_{\mathrm{I}}^{(1)})^2\right]^2\widehat{\psi}_1 \\
	&\qquad\quad\,+8i\lambda_{\mathrm{R}}^{(1)}\lambda_{\mathrm{I}}^{(1)}\left[(\lambda_{\mathrm{R}}^{(1)})^2+(\lambda_{\mathrm{I}}^{(1)})^2\right]\left(-\widehat{\psi}_0-\widehat{\psi}_1+i\xi\widehat{w}_0\right)\\
	&\qquad\quad\,-4i\lambda_{\mathrm{I}}^{(1)}\left[(\lambda_{\mathrm{R}}^{(1)})^2+(\lambda_{\mathrm{I}}^{(1)})^2\right]\left(\widehat{\psi}_0-i\xi\widehat{w}_0+i\xi\widehat{w}_1\right)\big|\\
	&\lesssim\chi_{\intt}(\xi)|\xi|^2\left(|\widehat{w}_0|+|\widehat{w}_1|+|\widehat{\psi}_0|+|\widehat{\psi}_1| \right).
\end{align*}
Thanks to the asymptotic expansions of characteristic roots stated in Subsection \ref{Subsection-Charcteristic-Roots}, there are some cancellations for the leading coefficients of $\widehat{\psi}_0$, $\widehat{\psi}_1$, $\widehat{w}_0$ so that
\begin{align*}
&\chi_{\intt}(\xi)\big|4i\lambda_{\mathrm{I}}^{(1)}\left[(\lambda_{\mathrm{R}}^{(1)})^2+(\lambda_{\mathrm{I}}^{(1)})^2\right]^2\widehat{\psi}_1-8i\lambda_{\mathrm{R}}^{(1)}\lambda_{\mathrm{I}}^{(1)}\left[(\lambda_{\mathrm{R}}^{(1)})^2+(\lambda_{\mathrm{I}}^{(1)})^2\right]\left(-\widehat{\psi}_0-\widehat{\psi}_1+i\xi\widehat{w}_0\right) \\
&\qquad\quad\,+4i\lambda_{\mathrm{I}}^{(1)}\left[(\lambda_{\mathrm{R}}^{(1)})^2+(\lambda_{\mathrm{I}}^{(1)})^2\right]\left(\widehat{\psi}_0-i\xi\widehat{w}_0+i\xi\widehat{w}_1\right)-(-2\sqrt{3}\,\xi\widehat{w}_1)\big|\\
&=\chi_{\intt}(\xi)\left|\big(2\sqrt{3}i+O(|\xi|^2)\big)\widehat{\psi}_1+\big(2\sqrt{3}i+O(|\xi|^2)\big)\left(-\widehat{\psi}_0-\widehat{\psi}_1+i\xi\widehat{w}_0\right)\right.\\
&\qquad\quad\quad\   \left.+\big(2\sqrt{3}i+O(|\xi|^2)\big)\left(\widehat{\psi}_0-i\xi\widehat{w}_0+i\xi\widehat{w}_1\right)+2\sqrt{3}\,\xi\widehat{w}_1\right|\\
&\lesssim\chi_{\intt}(\xi)|\xi|^2\left(|\widehat{w}_0|+|\widehat{w}_1|+|\widehat{\psi}_0|+|\widehat{\psi}_1| \right),
\end{align*}
which leads to
\begin{align*}
\chi_{\intt}(\xi)\left|\left(\det(\mb{V}_2^{(1)})-\det(\mb{V}_2^{(2)})\right)-(-2\sqrt{3}\,\xi\widehat{w}_1)\right|\lesssim\chi_{\intt}(\xi)|\xi|^2\left(|\widehat{w}_0|+|\widehat{w}_1|+|\widehat{\psi}_0|+|\widehat{\psi}_1| \right).
\end{align*}
Analogously to the proof of Proposition \ref{Prop-w-small} we immediately conclude
\begin{align*}
\chi_{\intt}(\xi)\left|\widehat{\psi}^{(2)}-i\xi\frac{\sin(c_a|\xi|^2t)}{c_a|\xi|^2}\,\mathrm{e}^{-\frac{1}{2}|\xi|^2t}\widehat{w}_1\right|\lesssim\chi_{\intt}(\xi)\,\mathrm{e}^{-c|\xi|^2t}\left(|\widehat{w}_0|+|\widehat{w}_1|+|\widehat{\psi}_0|+|\widehat{\psi}_1| \right) 
\end{align*}
as well as
\begin{align*}
\chi_{\intt}(\xi)|\widehat{\psi}^{(2)}|\lesssim \chi_{\intt}(\xi)\left(1+\frac{|\sin(c_a|\xi|^2t)|}{c_a|\xi|}\right)\mathrm{e}^{-c|\xi|^2t}\left(|\widehat{w}_0|+|\widehat{w}_1|+|\widehat{\psi}_0|+|\widehat{\psi}_1| \right). 
\end{align*}
Consequently, summarizing the derived estimates in the above our proof is complete.
\end{proof}
\begin{remark}\label{Rem-Fourier-Cancellation}
According to Propositions \ref{Prop-w-small} and \ref{Prop-psi-small}, we notice a cancellation between the leading terms of $i\xi\widehat{w}$ and $\widehat{\psi}$ for small frequencies such that
\begin{align}
\chi_{\intt}(\xi)|i\xi\widehat{w}-\widehat{\psi}|&\lesssim \chi_{\intt}(\xi)\left|i\xi\widehat{w}-\widehat{\ml{G}}(t,|\xi|)\left(i\xi\widehat{w}_1+|\xi|^2(\widehat{\psi}_0+\widehat{\psi}_1)\right)\right|+\chi_{\intt}(\xi)\left|i\xi\widehat{\ml{G}}(t,|\xi|)\widehat{w}_1-\widehat{\psi}\right|\notag\\
&\quad+\chi_{\intt}(\xi)|\xi|^2\widehat{\ml{G}}(t,|\xi|)|\widehat{\psi}_0+\widehat{\psi}_1|\notag\\
&\lesssim\chi_{\intt}(\xi)(1+|\xi|)\,\mathrm{e}^{-c|\xi|^2t}\left(|\widehat{w}_0|+|\widehat{w}_1|+|\widehat{\psi}_0|+|\widehat{\psi}_1| \right),\label{oo-Err-01}
\end{align}
in which the leading term $i\xi\widehat{\ml{G}}(t,|\xi|)\widehat{w}_1$ of $i\xi\widehat{w}$ and $\widehat{\psi}$ was compensated by each other. It reveals the pointwise estimate \eqref{Point-wise-Kawashima} from \cite[Proposition 2.1]{Ide-Haramoto-Kawashima=2008} with the diffusion kernel $\mathrm{e}^{-c|\xi|^2t}$ but without the singularity $|\xi|^{-1}$ as those in $i\xi\widehat{w}$ or $\widehat{\psi}$ for small frequencies.
\end{remark}

By the same procedure as the proof of Proposition \ref{Prop-w-large} associated with the expressions of initial data $\widehat{\psi}_2$ and $\widehat{\psi}_3$, we immediately claim the next result.
\begin{prop}\label{Prop-psi-large}
	Let $\xi\in\ml{Z}_{\extt}(N_0)$ with $N_0\gg 1$. Then, $\widehat{\psi}$ fulfills the following pointwise estimates in the Fourier space:
	\begin{align*}
		\chi_{\extt}(\xi)|\widehat{\psi}|\lesssim
		\chi_{\extt}(\xi)\left(|\xi|^{-1}|\widehat{w}_0|+|\xi|^{-2}|\widehat{w}_1|+|\widehat{\psi}_0|+|\xi|^{-1}|\widehat{\psi}_1|\right)\times \begin{cases}
			\mathrm{e}^{-ct}&\mbox{if}\ \ a=1,\\
			\mathrm{e}^{-c|\xi|^{-2}t}&\mbox{if}\ \ a\neq1.
		\end{cases}
	\end{align*}
\end{prop}
Following \eqref{Est-well-posed-01} and Proposition \ref{Prop-psi-large} one gets
\begin{align*}
\|\chi_{\extt}(D)\partial_t^k\psi(t,\cdot)\|_{H^{3-k}}\lesssim\|(w_0,w_1),(\psi_0,\psi_1)\|_{(H^2\times H^1)\times (H^3\times H^2) }
\end{align*}
for $k=0,\dots,3$, which justifies the well-posedness for $\psi$ in \eqref{Well-posedness} provided that $(w_0,w_1)\in (H^2\times H^1)$ and $(\psi_0,\psi_1)\in (H^3\times H^2)$.

\section{Large time asymptotic behavior of solutions in the $L^2$ norm}\label{Section_L2-norm}\setcounter{equation}{0}
\subsection{Preliminary on optimal growth/decay estimates}
\hspace{5mm}As a preparation, motivated by our desired factors $\chi_{\intt}(\xi)\widehat{\ml{G}}(t,|\xi|)$ and $\chi_{\intt}(\xi)\xi\widehat{\ml{G}}(t,|\xi|)$ in Propositions \ref{Prop-w-small} and \ref{Prop-psi-small}, respectively, we next derive large time optimal growth/decay estimates for the following time-dependent function:
\begin{align*}
\ml{I}(t;k):=\left\|
\chi_{\intt}(\xi)\frac{|\sin(c_a|\xi|^2t)|}{c_a|\xi|^{k}}\,\mathrm{e}^{-c|\xi|^2t} \right\|_{L^2}
\end{align*}
with $k=0,1,2$ and $c>0$. Notice that the factor $|\xi|^{-k}$ may exert a strong singularity (with different degree depending on $k$) as $|\xi|\to 0$. Thus, it is interesting to understand an interplay among the dissipative part $\mathrm{e}^{-c|\xi|^2t}$, the oscillating part $\sin(c_a|\xi|^2t)$, the singularity $|\xi|^{-k}$ (if $\xi\in\ml{Z}_{\intt}(\varepsilon_0)$ as well as $k=1,2$), and their contributions to optimal large time estimates.
\begin{lemma}\label{Lemma-optimal}
Let $k=0,1,2$. It satisfies the following optimal estimates:
\begin{align*}
\ml{I}(t;k)\approx t^{\frac{2k-1}{4}}
\end{align*}
for large time $t\gg1$, which decays if $k=0$ whereas grows polynomially if $k=1,2$.
\end{lemma}
\begin{proof}
Clearly, our forthcoming target is to estimate
\begin{align*}
[\ml{I}(t;k)]^2=C\int_0^{\varepsilon_0}|\sin(c_a|\xi|^2t)|^2\,|\xi|^{-2k}\,\mathrm{e}^{-2c|\xi|^2t}\,\mathrm{d}|\xi|
\end{align*}
for $k=0,1,2$. Note that one may generalize our estimates for a general parameter $k$.

For one thing, by setting an ansatz $\eta=|\xi|t^{1/2}$ one derives the upper bound estimate
\begin{align*}
[\ml{I}(t;k)]^2&\lesssim t^{k-\frac{1}{2}}\int_0^{\varepsilon_0t^{1/2}}|\sin(c_a\eta^2)|^2\,\eta^{-2k}\,\mathrm{e}^{-2c\eta^2}\,\mathrm{d}\eta\\
&\lesssim t^{k-\frac{1}{2}}\int_0^{+\infty}\left|\frac{\sin(c_a\eta^2)}{\eta^{k}}\right|^2\mathrm{e}^{-2c\eta^2}\,\mathrm{d}\eta\lesssim t^{\frac{2k-1}{2}},
\end{align*}
where we applied the useful inequality with any $\eta\in(0,+\infty)$ that
\begin{align*}
\left|\frac{\sin(c_a\eta^2)}{\eta^{k}}\right|^2\lesssim \left|\frac{\sin(c_a\eta^2)}{c_a\eta^2}\right|^{k}\lesssim 1\ \ \mbox{for}\ \ k=0,1,2.
\end{align*}
It will be used in Subsection \ref{Sub-new-linear} that $[\ml{I}(t;k)]^2\lesssim 1$ for $k=0,1,2$ if $t\ll 1$ due to the previous bounded estimate.

To guarantee the optimality of last upper bound estimate for large time, we then investigate its lower bound. Let us choose a small parameter $\alpha_0>0$ such that $\alpha_0^2<\pi/(4c_a)$. Thus, there exists a positive constant $C_0$ such that
\begin{align*}
1\geqslant |\sin(c_a|\xi|^2t)|\geqslant C_0>0\ \ \mbox{for any}\ \ |\xi|\in[\alpha_0t^{-\frac{1}{2}},2\alpha_0 t^{-\frac{1}{2}}],
\end{align*}
because of $c_a|\xi|^2t\in[\alpha_0^2c_a,4\alpha_0^2c_a]\subset(0,\pi)$ with $c_a>0$. We furthermore take large time $t\gg1$ so that $2\alpha_0t^{-\frac{1}{2}}<\varepsilon_0$ holds.  One may shrink the domain of the integral from $[0,\varepsilon_0]$ to $[\alpha_0t^{-1/2},2\alpha_0 t^{-1/2}]$, and obtain the lower bound estimate
\begin{align*}
[\ml{I}(t;k)]^2&\gtrsim\int_{\alpha_0t^{-1/2}}^{2\alpha_0t^{-1/2}}|\sin(c_a|\xi|^2t)|^2\,|\xi|^{-2k}\,\mathrm{e}^{-2c|\xi|^2t}\mathrm{d}|\xi|\\
&\gtrsim t^{k}\int_{\alpha_0t^{-1/2}}^{2\alpha_0t^{-1/2}}\mathrm{d}|\xi|\gtrsim t^{\frac{2k-1}{2}}
\end{align*}
for large time $t\gg1$, which implies our desired optimal estimates.
\end{proof}

\subsection{Proof of Theorems \ref{Thm-w} and \ref{Thm-psi}}
\noindent\textbf{Optimal Upper Bound Estimates}: Firstly, an application of Lemma \ref{Lemma-optimal} with $k=0,2$ and Proposition \ref{Prop-w-small} yields
\begin{align*}
\|\chi_{\intt}(D)w(t,\cdot)\|_{L^2}&\lesssim\left\|\chi_{\intt}(\xi)\left(1+\frac{|\sin(c_a|\xi|^2t)|}{c_a|\xi|^2}\right)\mathrm{e}^{-c|\xi|^2t}\right\|_{L^2}\|(\widehat{w}_0,\widehat{w}_1),(\widehat{\psi}_0,\widehat{\psi}_1)\|_{(L^{\infty}\times L^{\infty})\times (L^{\infty}\times L^{\infty})}\\
&\lesssim t^{\frac{3}{4}}\|(w_0,w_1),(\psi_0,\psi_1)\|_{(L^1\times L^1)\times(L^1\times L^1)}
\end{align*}
for large time $t\gg1$, where the Hausdorff-Young inequality and the Plancherel theorem were used. Then, from Proposition \ref{Prop-w-large} one finds
\begin{align*}
\|\chi_{\extt}(D)w(t,\cdot)\|_{L^2}\lesssim \|(w_0,w_1),(\psi_0,\psi_1)\|_{(L^2\times H^{-1})\times( H^{-1}\times H^{-2})}
\end{align*}
for any $a>0$ due to the fact that $\mathrm{e}^{-ct},\mathrm{e}^{-c|\xi|^{-2}t}\leqslant 1$. Additionally from an exponential decay estimate for bounded frequencies via the negative roots \eqref{Negative-roots}, the following upper bound estimate holds:
\begin{align}\label{Est-02}
\|w(t,\cdot)\|_{L^2}&\lesssim \|\chi_{\intt}(D)w(t,\cdot)\|_{L^2}+\|\chi_{\bdd}(D)w(t,\cdot)\|_{L^2}+\|\chi_{\extt}(D)w(t,\cdot)\|_{L^2}\notag\\
&\lesssim t^{\frac{3}{4}}\|(w_0,w_1),(\psi_0,\psi_1)\|_{(\ml{D}_0^0\times\ml{D}_0^{-1})\times(\ml{D}_0^{-1}\times\ml{D}_0^{-2})}
\end{align}
for large time $t\gg1$. By the same way (i.e. Lemma \ref{Lemma-optimal} with $k=0,1$ and Proposition \ref{Prop-psi-small} for small frequencies; Proposition \ref{Prop-psi-large} for large frequencies) we are able to derive
\begin{align}
\|\psi(t,\cdot)\|_{L^2}\lesssim t^{\frac{1}{4}}\|(w_0,w_1),(\psi_0,\psi_1)\|_{(\ml{D}_0^{-1}\times\ml{D}_0^{-2})\times(\ml{D}_0^0\times \ml{D}_0^{-1})}\label{Est-03}
\end{align}
for large time $t\gg1$.\medskip

\noindent\textbf{Optimal Lower Bound Estimates}: To ensure the sharpness of \eqref{Est-02} and \eqref{Est-03}, we next estimate them from below. Indeed, \eqref{Error-w02} implies
\begin{align*}
\chi_{\intt}(\xi)|\widehat{w}-\widehat{\ml{G}}(t,|\xi|)\widehat{w}_1|\lesssim \chi_{\intt}(\xi)\left(1+\frac{|\sin(c_a|\xi|^2t)|}{c_a|\xi|}\right)\mathrm{e}^{-c|\xi|^2t}\left(|\widehat{w}_0|+|\widehat{w}_1|+|\widehat{\psi}_0|+|\widehat{\psi}_1| \right).
\end{align*}
 By the same manner as the last parts associated with the previous pointwise estimate in the Fourier space and Proposition \ref{Prop-w-large}, one obtains
\begin{align*}
\left\|\big(1-\chi_{\intt}(D)\big)\ml{\ml{G}}(t,|D|)w_1(\cdot)\right\|_{L^2}&\lesssim\mathrm{e}^{-ct}\|w_1\|_{H^{-2}},\\
\|w(t,\cdot)-\ml{G}(t,|D|)w_1(\cdot)\|_{L^2}&\lesssim t^{\frac{1}{4}}\|(w_0,w_1),(\psi_0,\psi_1)\|_{(\ml{D}_0^0\times\ml{D}_0^{-1})\times(\ml{D}_0^{-1}\times\ml{D}_0^{-2})},
\end{align*}
for large time $t\gg1$. Let us recall
\begin{align*}
\ml{F}_{x\to\xi}\big(G(t,x)\big)=\frac{\sin(c_a|\xi|^2t)}{c_a|\xi|^2}\,\mathrm{e}^{-\frac{1}{2}|\xi|^2t}=\widehat{\ml{G}}(t,|\xi|).
\end{align*}
 By the mean value theorem
\begin{align*}
|G(t,x-y)-G(t,x)|\lesssim |y|\,|\partial_xG(t,x-\eta_1y)| \ \ \mbox{with}\ \ \eta_1\in(0,1),
\end{align*}
and Lemma \ref{Lemma-optimal} with $k=1,2$ for small frequencies, the next estimate for large frequencies:
\begin{align*}
\|\,|\xi|^k\widehat{\ml{G}}(t,|\xi|)\|_{L^2(|\xi|\geqslant 1)}^2&\lesssim\int_1^{+\infty}|\xi|^{2(k-2)}\,\mathrm{e}^{-|\xi|^2t}\,\mathrm{d}|\xi|\\
&\lesssim\mathrm{e}^{-ct}\int_1^{+\infty}|\xi|^{2(k-2)}\,\mathrm{d}|\xi|\lesssim \mathrm{e}^{-ct}\ \ \mbox{with}\ \ k=0,1,
\end{align*}
one may separate the integral into two parts such that
\begin{align*}
&\|\ml{G}(t,|D|)w_1(\cdot)-G(t,\cdot)P_{w_1}\|_{L^2}\\
&\lesssim\left\|\int_{|y|\leqslant t^{1/8}}\big(G(t,\cdot-y)-G(t,\cdot)\big)w_1(y)\,\mathrm{d}y\,\right\|_{L^2}+\left\|\int_{|y|\geqslant t^{1/8}}\big(|G(t,\cdot-y)|+|G(t,\cdot)|\big)|w_1(y)|\,\mathrm{d}y\,\right\|_{L^2}\\
&\lesssim t^{\frac{1}{8}}\|\,|\xi|\widehat{\ml{G}}(t,|\xi|)\|_{L^2}\|w_1\|_{L^1}+\|\widehat{\ml{G}}(t,|\xi|)\|_{L^2}\|w_1\|_{L^1(|x|\geqslant t^{1/8})}\\
&\lesssim t^{\frac{1}{8}}(t^{\frac{1}{4}}+\mathrm{e}^{-ct})\|w_1\|_{L^1}+(t^{\frac{3}{4}}+\mathrm{e}^{-ct})\,o(1)
\end{align*}
for large time $t\gg1$, thanks to our assumption on the $L^1$ integrability of $w_1$ leading to
\begin{align*}
	\lim\limits_{t\to+\infty}\int_{|x|\geqslant t^{1/8}}|w_1(x)|\,\mathrm{d}x=0.
\end{align*}
In other words, we conclude
\begin{align*}
\|w(t,\cdot)-G(t,\cdot)P_{w_1}\|_{L^2}&\leqslant \|w(t,\cdot)-\ml{G}(t,|D|)w_1(\cdot)\|_{L^2}+\|\ml{G}(t,|D|)w_1(\cdot)-G(t,\cdot)P_{w_1}\|_{L^2}\\
&\lesssim t^{\frac{1}{4}}\|(w_0,w_1),(\psi_0,\psi_1)\|_{(\ml{D}_0^0\times\ml{D}_0^{-1})\times(\ml{D}_0^{-1}\times\ml{D}_0^{-2})}+o(t^{\frac{3}{4}})\\
&=o(t^{\frac{3}{4}}),
\end{align*}
and, similarly from Propositions \ref{Prop-psi-small} and \ref{Prop-psi-large},
\begin{align}\label{Error-psi-weak}
\|\psi(t,\cdot)-\partial_xG(t,\cdot)P_{w_1}\|_{L^2}=o(t^{\frac{1}{4}})
\end{align}
for large time $t\gg1$. In the above error estimate of $\psi$, we employed
\begin{align*}
\left\|\big(1-\chi_{\intt}(D)\big)\partial_x\ml{G}(t,|D|)w_1(\cdot)\right\|_{L^2}&\lesssim\left\|\big(1-\chi_{\intt}(\xi)\big)|\xi|^{-1}\,\mathrm{e}^{-c|\xi|^2t}\,\widehat{w}_1(\xi) \right\|_{L^2}\\
&\lesssim\mathrm{e}^{-ct}\, t^{-\frac{1}{2}}\sup\limits_{|\xi|\geqslant \varepsilon_0}\left((|\xi|^2t)^{\frac{1}{2}}\,\mathrm{e}^{-c|\xi|^2t}\right)\|\langle\xi\rangle^{-2}\widehat{w}_1(\xi)\|_{L^2}\\
&\lesssim \mathrm{e}^{-ct}\|w_1\|_{H^{-2}}.
\end{align*}
Eventually, let us employ the Minkowski inequality and combine the last estimates with Lemma \ref{Lemma-optimal} to arrive at
\begin{align*}
\|w(t,\cdot)\|_{L^2}&\gtrsim\|G(t,\cdot)\|_{L^2}|P_{w_1}|-\|w(t,\cdot)-G(t,\cdot)P_{w_1}\|_{L^2}\gtrsim t^{\frac{3}{4}}|P_{w_1}|-o(t^{\frac{3}{4}}),\\
\|\psi(t,\cdot)\|_{L^2}&\gtrsim\|\partial_xG(t,\cdot)\|_{L^2}|P_{w_1}|-\|\psi(t,\cdot)-\partial_xG(t,\cdot)P_{w_1}\|_{L^2}\gtrsim t^{\frac{1}{4}}|P_{w_1}|-o(t^{\frac{1}{4}}),
\end{align*}
for large time $t\gg1$. Thanks to our non-trivial assumption $P_{w_1}\neq 0$, they immediately provide optimal lower bound estimates.\medskip 

\noindent\textbf{Large Time Profile of $\psi$}: Due to the additional hypothesis $w_1\in L^{1,1}$ and \cite[Lemma 3.1]{Ikehata=2004} (see also \cite[Lemma 2.2]{Ikehata=2014} and \cite[Lemma 5.1 with $\gamma=1$]{Ikehata-Michihisa=2019}) we may deduce
\begin{align}\label{Tools-xi}
|\widehat{w}_1-P_{w_1}|\lesssim|\xi|\,\|w_1\|_{L^{1,1}},
\end{align}
so that for large time $t\gg1$ \eqref{Error-psi-weak} can be improved by
\begin{align*}
&\|\psi(t,\cdot)-\partial_xG(t,\cdot)P_{w_1}\|_{L^2}\\
&\lesssim\|\psi(t,\cdot)-\ml{G}(t,|D|)\mathrm{d}_xw_1(\cdot)\|_{L^2}+\|\ml{G}(t,|D|)\mathrm{d}_xw_1(\cdot)-\partial_xG(t,\cdot)P_{w_1}\|_{L^2}\\
&\lesssim \left\|\chi_{\intt}(\xi)\,\mathrm{e}^{-c|\xi|^2t}\right\|_{L^2}\|(w_0,w_1),(\psi_0,\psi_1)\|_{(L^1\times L^1)\times(L^1\times L^1)}+\left\|\big(1-\chi_{\intt}(\xi)\big)\widehat{\psi}(t,\xi)\right\|_{L^2}\\
&\quad+\left\|\big(1-\chi_{\intt}(\xi)\big)\xi\widehat{\ml{G}}(t,|\xi|)\widehat{w}_1(\xi)\right\|_{L^2}+\left\||\xi|^2\widehat{\ml{G}}(t,|\xi|)\right\|_{L^2}\|w_1\|_{L^{1,1}}\\
&\lesssim \begin{cases}
	t^{-\frac{1}{4}}\|(w_0,w_1),(\psi_0,\psi_1)\|_{(\ml{D}_0^{-1}\times\ml{D}_1^{-2})\times(\ml{D}_0^0\times\ml{D}_0^{-1})}&\mbox{if}\ \ a=1,\\
	\max\{t^{-\frac{1}{4}},t^{-\frac{\ell}{2}}\}\|(w_0,w_1),(\psi_0,\psi_1)\|_{(\ml{D}_0^{\ell-1}\times\ml{D}_1^{\ell-2})\times(\ml{D}_0^{\ell}\times\ml{D}_0^{\ell-1})}&\mbox{if}\ \ a\neq1.
\end{cases}
\end{align*}
In the above, we used the following sharp regularity-loss type decay estimate to treat $\widehat{\psi}$ for large frequencies:
\begin{align}\label{Tools-Large-Freq}
	\left\|\chi_{\extt}(\xi)\,\mathrm{e}^{-c|\xi|^{-2}t}\,\widehat{g}(\xi)\right\|_{L^2}&\lesssim\left\|\chi_{\extt}(\xi)|\xi|^{-\ell}\,\mathrm{e}^{-c|\xi|^{-2}t}\right\|_{L^{\infty}}\|\langle\xi\rangle^{\ell}\widehat{g}(\xi)\|_{L^2}\notag\\
	&\lesssim t^{-\frac{\ell}{2}}\|g\|_{H^{\ell}}
\end{align}
for a given function (taken as the initial data) $g=g(x)\in H^{\ell}$ with any $\ell>0$. Then, one may claim the desired refined estimate for $\psi$.\medskip

\noindent\textbf{Large Time Profile of $w$}: Let us concentrate on the further error estimate by the next suitable decomposition:
\begin{align*}
w(t,x)-\big(G(t,x)P_{w_1}-\partial_x G(t,x) Q_{w_1}-\partial_xG(t,x)P_{\psi_0+\psi_1}\big)=\sum\limits_{k=1,2,3}\ml{E}_k(t,x)
\end{align*}
whose components on the right-hand side are denoted by
\begin{align*}
\ml{E}_1(t,x)&:=w(t,x)-\ml{G}(t,|D|)w_1(x)+\ml{G}(t,|D|)\,\mathrm{d}_x\big(\psi_0(x)+\psi_1(x)\big),\\
\ml{E}_2(t,x)&:=\ml{G}(t,|D|)w_1(x)-G(t,x)P_{w_1}+\partial_xG(t,x) Q_{w_1},\\
\ml{E}_3(t,x)&:=-\ml{G}(t,|D|)\,\mathrm{d}_x\big(\psi_0(x)+\psi_1(x)\big)+\partial_xG(t,x)P_{\psi_0+\psi_1}.
\end{align*}
According to Propositions \ref{Prop-w-small} as well as \ref{Prop-w-large}, one employs \eqref{Tools-Large-Freq} and Lemma \ref{Lemma-optimal} to derive
\begin{align*}
	\|\ml{E}_1(t,\cdot)\|_{L^2}&=\left\|\widehat{w}(t,\xi)-\widehat{\ml{G}}(t,|\xi|)\widehat{w}_1(\xi)+i\xi\widehat{\ml{G}}(t,|\xi|)\big(\widehat{\psi}_0(\xi)+\widehat{\psi}_1(\xi)\big)\right\|_{L^2}\\
	&\lesssim \left\|\chi_{\intt}(\xi)\,\mathrm{e}^{-c|\xi|^2t}\right\|_{L^2}\|(w_0,w_1),(\psi_0,\psi_1)\|_{(L^1\times L^1)\times (L^1\times L^1)}+\left\|\big(1-\chi_{\intt}(\xi)\big)\widehat{w}(t,\xi)\right\|_{L^2}\\
	&\quad+\left\|\big(1-\chi_{\intt}(\xi)\big)\widehat{\ml{G}}(t,|\xi|)\widehat{w}_1(\xi)\right\|_{L^2}+\left\|\big(1-\chi_{\intt}(\xi)\big)\xi\widehat{\ml{G}}(t,|\xi|)\big(\widehat{\psi}_0(\xi)+\widehat{\psi}_1(\xi)\big)\right\|_{L^2}\\
	&\lesssim \begin{cases}
		t^{-\frac{1}{4}}\|(w_0,w_1),(\psi_0,\psi_1)\|_{(\ml{D}_0^0\times\ml{D}_0^{-1})\times(\ml{D}_0^{-1}\times\ml{D}_0^{-2})}&\mbox{if}\ \ a=1,\\
		\max\{t^{-\frac{1}{4}},t^{-\frac{\ell}{2}}\}\|(w_0,w_1),(\psi_0,\psi_1)\|_{(\ml{D}_0^{\ell}\times\ml{D}_0^{\ell-1})\times(\ml{D}_0^{\ell-1}\times\ml{D}_0^{\ell-2})}&\mbox{if}\ \ a\neq1,
	\end{cases}
\end{align*}
for large time $t\gg1$. 
Moreover, thanks to our assumption $\psi_0+\psi_1\in L^{1,1}$ in the improved estimate \eqref{Tools-xi}, as $t\gg1$ the third error term is estimated by
\begin{align*}
\|\ml{E}_3(t,\cdot)\|_{L^2}\lesssim\left\||\xi|^2\widehat{\ml{G}}(t,|\xi|)\right\|_{L^2}\|\psi_0+\psi_1\|_{L^{1,1}}\lesssim t^{-\frac{1}{4}}\|(\psi_0,\psi_1)\|_{L^{1,1}\times L^{1,1}}.
\end{align*}
In order to treat the final error term $\ml{E}_2(t,x)$, we recall \cite[Lemma 5.1 with $\gamma=2$]{Ikehata-Michihisa=2019} to address
\begin{align*}
|\widehat{w}_1-P_{w_1}+i\xi Q_{w_1}|&\leqslant\left|\widehat{w}_1-P_{w_1}+i\xi Q_{w_1}+\frac{|\xi|^2}{2}\int_{\mb{R}}|x|^2w_1(x)\,\mathrm{d}x\right|+\frac{|\xi|^2}{2}\int_{\mb{R}}|x|^2|w_1(x)|\,\mathrm{d}x\\
&\lesssim|\xi|^2\int_{\mb{R}}|x|^2|w_1(x)|\,\mathrm{d}x\lesssim |\xi|^2\|w_1\|_{L^{1,2}}.
\end{align*}
One then deduces
\begin{align*}
\|\ml{E}_2(t,\cdot)\|_{L^2}&=\left\|\widehat{\ml{G}}(t,|\xi|)\widehat{w}_1(\xi)-\widehat{\ml{G}}(t,|\xi|)P_{w_1}+i\xi\widehat{\ml{G}}(t,|\xi|)Q_{w_1}\right\|_{L^2}\\
&\lesssim\left\||\xi|^2\widehat{\ml{G}}(t,|\xi|)\right\|_{L^2}\|w_1\|_{L^{1,2}}\lesssim t^{-\frac{1}{4}}\|w_1\|_{L^{1,2}}.
\end{align*}
Summarizing all estimates in the above we conclude our desired estimate \eqref{Error-w}.

\section{Global in time existence for the semilinear Cauchy problem}\label{Section-GESDS}\setcounter{equation}{0}
\subsection{Philosophy of our proof}
\hspace{5mm}For any $T_*>0$ we introduce the evolution space $\ml{X}_{T_*}$ of solution $\ml{U}=\ml{U}(t,x)$ in the vector sense that $\ml{U}:=(w^N,\psi^N)^{\mathrm{T}}$ by
\begin{align*}
\ml{X}_{T_*}:=
\big(\ml{C}([0,{T_*}],H^2)\times \ml{C}([0,{T_*}],H^{1})\big)^{\mathrm{T}}
\end{align*}
equipped the time-weighted norm
\begin{align*}
\|\ml{U}\|_{\ml{X}_{T_*}}:=\sup\limits_{t\in[0,T_*]}\left(\sum\limits_{k=0,2}(1+t)^{\frac{-1+2k}{4}}\|w^N(t,\cdot)\|_{\dot{H}^k}+\sum\limits_{k=0,1}(1+t)^{\frac{1+2k}{4}}\|\psi^N(t,\cdot)\|_{\dot{H}^k}\right).
\end{align*}
Its time-dependent weighted functions are strongly motivated by some estimates of solutions to the corresponding linearized Cauchy problem \eqref{Dissipative-Timoshenko} with the equal wave speeds $a=1$ as well as $P_{w_1}=0$, whose details will be showed in the next subsection.

Let us define the linear part $u_{\lin}=u_{\lin}(t,x)$ with $u=w,\psi$ by
\begin{align*}
	u_{\lin}(t,x):=K_0^{u}(t,\partial_x)w_0(x)+K_1^u(t,\partial_x)w_1(x)+K_2^u(t,\partial_x)\psi_0(x)+K_3^u(t,\partial_x)\psi_1(x),
\end{align*}
where the kernels $K_j^u(t,\partial_x)$ can be defined in Section \ref{Section_Fourier_Space} by the Fourier transform. Then, the unknown vector $\ml{U}_{\lin}=\ml{U}_{\lin}(t,x)$ to the corresponding linearized dissipative Timoshenko system to \eqref{Nonlinear-Dissipative-Timoshenko} with vanishing right-hand sides is expressed by
\begin{align*}
\ml{U}_{\lin}(t,x)&:=\left(
{\begin{array}{*{20}c}
	w_{\lin}(t,x)\\
	\psi_{\lin}(t,x)
\end{array}}
\right)=
\left(
{\begin{array}{*{20}c}
	K_0^w(t,\partial_x) & K_2^w(t,\partial_x)\\
	K_0^{\psi}(t,\partial_x) & K_2^{\psi}(t,\partial_x)
\end{array}}
\right)
\ml{U}_{\lin}(0,x)+\left(
{\begin{array}{*{20}c}
	K_1^w(t,\partial_x) & K_3^w(t,\partial_x)\\
	K_1^{\psi}(t,\partial_x) & K_3^{\psi}(t,\partial_x)
\end{array}}
\right)
\partial_t\ml{U}_{\lin}(0,x)
\end{align*}
with the initial conditions
\begin{align*}
\ml{U}_{\lin}(0,x):=\left(
{\begin{array}{*{20}c}
	w_0^N(x)\\
	\psi_0^N(x)
\end{array}}
\right)
\ \ \mbox{as well as}\ \ \partial_t\ml{U}_{\lin}(0,x):=\left(
{\begin{array}{*{20}c}
	w_1^N(x)\\
	\psi_1^N(x)
\end{array}}
\right).
\end{align*}

From Duhamel’s principle, concerning the semilinear Cauchy problem \eqref{Nonlinear-Dissipative-Timoshenko} in the next vector version of second order evolution system (i.e. semilinear wave system with lower order terms):
\begin{align*}
\partial_t^2\ml{U}-
\left(
{\begin{array}{*{20}c}
		1 & 0  \\
		0  & 1 \\
\end{array}}
\right)
\partial_x^2\ml{U}-\left(
{\begin{array}{*{20}c}
		0 & -1  \\
		1  & 0 \\
\end{array}}
\right)\partial_x\ml{U}+\left(
{\begin{array}{*{20}c}
		0 & 0  \\
		0  & 1 \\
\end{array}}
\right)\partial_t\ml{U}+\left(
{\begin{array}{*{20}c}
		0 & 0  \\
		0  & 1 \\
\end{array}}
\right)\ml{U}=\left(
{\begin{array}{*{20}c}
0\\
|\psi^N|^p
\end{array}}
\right),
\end{align*}
 let us construct the following nonlinear integral operator:
\begin{align*}
\ml{N}:\ \ml{U}\in \ml{X}_{T_*}\to \ml{N}[\ml{U}]:=\ml{U}_{\lin}+\ml{U}_{\nlin},
\end{align*}
with $\ml{U}_{\nlin}=\ml{U}_{\nlin}(t,x)$ defined by
\begin{align*}
\ml{U}_{\nlin}(t,x)=\left(
{\begin{array}{*{20}c}
		w_{\nlin}(t,x)\\
		\psi_{\nlin}(t,x)
\end{array}}
\right)&:=\int_0^t\left(
{\begin{array}{*{20}c}
	K_1^w(t-\eta,\partial_x) & K_3^w(t-\eta,\partial_x)\\
	K_1^{\psi}(t-\eta,\partial_x) & K_3^{\psi}(t-\eta,\partial_x)
\end{array}}
\right)
\left(
{\begin{array}{*{20}c}
0\\
|\psi^N(\eta,x)|^p
\end{array}}
\right)
\mathrm{d}\eta\\
&\ =\int_0^t\left(
{\begin{array}{*{20}c}
	K_3^w(t-\eta,\partial_x)|\psi^N(\eta,x)|^p\\
	K_3^{\psi}(t-\eta,\partial_x)|\psi^N(\eta,x)|^p	
\end{array}}
\right)\mathrm{d}\eta.
\end{align*}

To achieve our goal, i.e. $\ml{U}$ is a fixed point of the operator $\ml{N}$ in $\ml{X}_{+\infty}$ by applying the Banach contraction principle, we have to show that the following fundamental inequalities:
\begin{align}\label{Crucial-01}
\|\ml{N}[\ml{U}]\|_{\ml{X}_{T_*}}&\lesssim\|(w_0^N,w_1^N),(\psi_0^N,\psi_1^N)\|_{(\ml{D}_0^2\times\ml{D}_1^{1})\times(\ml{D}_0^{1}\times\ml{D}_0^{0})}+\|\ml{U}\|_{\ml{X}_{T_*}}^p,\\
\|\ml{N}[\ml{U}]-\ml{N}[\widetilde{\ml{U}}]\|_{\ml{X}_{T_*}}&\lesssim\|\ml{U}-\widetilde{\ml{U}}\|_{\ml{X}_{T_*}}\left(\|\ml{U}\|_{\ml{X}_{T_*}}^{p-1}+\|\widetilde{\ml{U}}\|_{\ml{X}_{T_*}}^{p-1} \right),\label{Crucial-02}
\end{align}
hold for any $\ml{U},\widetilde{\ml{U}}\in \ml{X}_{T_*}$, where all unexpressed multiplicative constants are independent of $T_*$. Later, the estimates \eqref{Linear-01}, \eqref{Linear-02} will show $\ml{U}_{\lin}\in \ml{X}_{T_*}$ for any $T_*>0$ and the uniformly bounded estimate
\begin{align}\label{Crucial-03}
\|\ml{U}_{\lin}\|_{\ml{X}_{T_*}}\lesssim \|(w_0^N,w_1^N),(\psi_0^N,\psi_1^N)\|_{(\ml{D}_0^2\times\ml{D}_1^{1})\times(\ml{D}_0^{1}\times\ml{D}_0^{0})}.
\end{align}
Remark that the combination of \eqref{Crucial-03} and \eqref{Crucial-02} with $\widetilde{\ml{U}}\equiv0$ follows \eqref{Crucial-01}. In other words, we need to justify \eqref{Crucial-02} under the condition $p>p_{\mathrm{Fuj}}(1)=3$ in the forthcoming part.

\subsection{Preliminary on the linearized Cauchy problem}\label{Sub-new-linear}
\hspace{5mm}As our preparation, we in this subsection derive some sharp estimates for the corresponding linearized Cauchy problem \eqref{Dissipative-Timoshenko} with the equal wave speeds (i.e. $a=1$) when $P_{w_1}=0$. From \eqref{Tools-xi} with $P_{w_1}=0$ it holds that
\begin{align*}
\chi_{\intt}(\xi)\widehat{\ml{G}}(t,|\xi|)|\widehat{w}_1|\lesssim \chi_{\intt}(\xi)|\xi|\widehat{\ml{G}}(t,|\xi|)\|w_1\|_{L^{1,1}}
\end{align*}
for $k=1,2$. Therefore, recalling the derived pointwise estimates \eqref{Error-w02} as well as \eqref{Error-psi02}, one knows
\begin{align*}
\chi_{\intt}(\xi)|\widehat{w}|&\lesssim\chi_{\intt}(\xi)\,\mathrm{e}^{-c|\xi|^2t}|\widehat{w}_0|+\chi_{\intt}(\xi)\left(1+\frac{|\sin(c_a|\xi|^2t)|}{c_a|\xi|}\right)\mathrm{e}^{-c|\xi|^2t}\left(\|w_1\|_{L^{1,1}}+|\widehat{\psi}_0|+|\widehat{\psi}_1|\right),\\
\chi_{\intt}(\xi)|\widehat{\psi}|&\lesssim\chi_{\intt}(\xi)\,\mathrm{e}^{-c|\xi|^2t}\left(|\widehat{w}_0|+\|w_1\|_{L^{1,1}}+|\widehat{\psi}_0|+|\widehat{\psi}_1|\right).
\end{align*}
Let us apply the last two estimates associated with Propositions \ref{Prop-w-large} and \ref{Prop-psi-large}. They yield
\begin{align*}
\|w(t,\cdot)\|_{L^2}&\lesssim (1+t)^{-\frac{1}{4}}\|w_0\|_{L^2\cap L^1}+(1+t)^{\frac{1}{4}}\|(w_1,\psi_0,\psi_1)\|_{(H^{-1}\cap L^{1,1})\times (H^{-1}\cap L^1)\times (H^{-2}\cap L^1)},\\
\|w(t,\cdot)\|_{\dot{H}^2}&\lesssim (1+t)^{-\frac{5}{4}}\|w_0\|_{H^2\cap L^1}+(1+t)^{-\frac{3}{4}}\|(w_1,\psi_0,\psi_1)\|_{(H^{1}\cap L^{1,1})\times (H^{1}\cap L^1)\times (L^2\cap L^1)},
\end{align*}
where we used Lemma \ref{Lemma-optimal}. Moreover, by choosing $k=0,1$, we arrive at
\begin{align*}
\|\psi(t,\cdot)\|_{\dot{H}^k}&\lesssim \left\|\chi_{\intt}(\xi)|\xi|^k\,\mathrm{e}^{-c|\xi|^2t}\right\|_{L^2}\|(w_0,w_1),(\psi_0,\psi_1)\|_{(L^1\times L^{1,1})\times (L^1\times L^1)}\\
&\quad+
\mathrm{e}^{-ct}\|(w_0,w_1),(\psi_0,\psi_1)\|_{(H^{k-1}\times H^{k-2})\times (H^k\times H^{k-1})}\\
&\lesssim (1+t)^{-\frac{1}{4}-\frac{k}{2}}
	\|(w_0,w_1),(\psi_0,\psi_1)\|_{(H^{k-1}\cap L^1)\times (H^{k-2}\cap L^{1,1})\times (H^k\cap L^1)\times (H^{k-1}\cap L^1)},
\end{align*}
where we used
\begin{align*}
\left\|\chi_{\intt}(\xi)|\xi|^k\,\mathrm{e}^{-c|\xi|^2t}\right\|_{L^2}^2=\int_0^{\varepsilon_0}|\xi|^{2k}\,\mathrm{e}^{-2c|\xi|^2t}\,\mathrm{d}|\xi|\lesssim (1+t)^{-\frac{1}{2}-k}.
\end{align*}
In conclusion, the solutions to the linearized dissipative Timoshenko system \eqref{Dissipative-Timoshenko} with $a=1$ and $P_{w_1}=0$ satisfy
\begin{align}
\|w(t,\cdot)\|_{\dot{H}^k}&\lesssim (1+t)^{\frac{1}{4}-\frac{k}{2}}\|(w_0,w_1),(\psi_0,\psi_1)\|_{(\ml{D}_0^2\times\ml{D}_1^{1})\times(\ml{D}_0^{1}\times\ml{D}_0^{0})}\ \ \ \ \,\, \mbox{for}\ \ k=0,2,\label{Linear-01}\\
\|\psi(t,\cdot)\|_{\dot{H}^k}&\lesssim(1+t)^{-\frac{1}{4}-\frac{k}{2}}\|(w_0,w_1),(\psi_0,\psi_1)\|_{(\ml{D}_0^0\times\ml{D}_1^{-1})\times(\ml{D}_0^{1}\times\ml{D}_0^{0})}\ \ \mbox{for}\ \ k=0,1.\label{Linear-02}
\end{align}
The additional condition $P_{w_1}=0$ may improve the time-dependent coefficients of the estimates \eqref{Est-02} and \eqref{Est-03} by $t^{-1/2}$ as large time $t\gg1$, which is caused by the gained factor $|\xi|$ for small frequencies from \eqref{Tools-xi}.
Note that since $\ml{D}_0^0=L^2\cap L^1$ they also imply the $(L^2\cap L^1)-\dot{H}^k$ estimates
\begin{align}
\|K_3^{w}(t-\eta,\partial_x)g\|_{\dot{H}^k}&\lesssim (1+t-\eta)^{\frac{1}{4}-\frac{k}{2}}\|g\|_{L^2\cap L^1}\ \ \, \ \mbox{for}\ \ k=0,2,\label{Help-03}\\
\|K_3^{\psi}(t-\eta,\partial_x)g\|_{\dot{H}^k}&\lesssim (1+t-\eta)^{-\frac{1}{4}-\frac{k}{2}}\|g\|_{L^2\cap L^1}\ \ \mbox{for}\ \ k=0,1.\label{Help-01}
\end{align}

Then, we also can obtain the $L^2-\dot{H}^k$ estimates
\begin{align}
\|K_3^{w}(t-\eta,\partial_x)g\|_{\dot{H}^2}&\lesssim (1+t-\eta)^{-\frac{1}{2}}\|g\|_{L^2},\label{Help-04}\\
	\|K_3^{\psi}(t-\eta,\partial_x)g\|_{\dot{H}^k}&\lesssim (1+t-\eta)^{-\frac{k}{2}}\|g\|_{L^2}\ \ \mbox{for}\ \ k=0,1,\label{Help-02}
\end{align}
by using $\|\chi_{\intt}(\xi)|\xi|^k\,\mathrm{e}^{-c|\xi|^2t} \|_{L^{\infty}}\lesssim (1+t)^{-\frac{k}{2}}$.

\subsection{Proof of Theorem \ref{Thm-GESDS}}
\hspace{5mm}
 We first rewrite the difference of nonlinear terms by
\begin{align*}
\left||\psi^N(\eta,x)|^p-|\widetilde{\psi}^N(\eta,x)|^p\right|&=\left|\int_0^1\frac{\partial}{\partial\varrho}\left|\varrho\,\psi^N(\eta,x)+(1-\varrho)\widetilde{\psi}^N(\eta,x)\right|^p\mathrm{d}\varrho\,\right|\\
&\lesssim\left(|\psi^N(\eta,x)|^{p-1}+|\widetilde{\psi}^N(\eta,x)|^{p-1} \right) |\psi^N(\eta,x)-\widetilde{\psi}^N(\eta,x)|.
\end{align*}
Applying H\"older's inequality one arrives at
\begin{align*}
\left\||\psi^N(\eta,\cdot)|^p-|\widetilde{\psi}^N(\eta,\cdot)|^p\right\|_{L^r}\lesssim\|\psi^{N}(\eta,\cdot)-\widetilde{\psi}^N(\eta,\cdot)\|_{L^{rp}}\left(\|\psi^N(\eta,\cdot)\|_{L^{rp}}^{p-1}+\|\widetilde{\psi}^N(\eta,\cdot)\|_{L^{rp}}^{p-1}\right)
\end{align*}
for any $r\geqslant 1$. Thanks to the norm of $\ml{X}_{T_*}$, an application of the well-known Gagliardo-Nirenberg inequality shows
\begin{align*}
\|\psi^N(\eta,\cdot)\|_{L^{rp}}&\lesssim \|\psi^N(\eta,\cdot)\|_{L^2}^{\frac{1}{2}+\frac{1}{rp}}\|\psi^N(\eta,\cdot)\|_{\dot{H}^{1}}^{\frac{1}{2}-\frac{1}{rp}}\lesssim (1+\eta)^{-\frac{1}{2}+\frac{1}{2rp}}\|\ml{U}\|_{\ml{x}_\eta},
\end{align*}
where we took $p\geqslant 2/r$ from the restriction of the Gagliardo-Nirenberg inequality. Similarly, with any $p\geqslant 2/r$ and $r\geqslant 1$, for $\eta\leqslant T_*$ we can conclude
\begin{align}
\left\||\psi^N(\eta,\cdot)|^p-|\widetilde{\psi}^N(\eta,\cdot)|^p\right\|_{L^r}\lesssim(1+\eta)^{-\frac{p}{2}+\frac{1}{2r}}\|\ml{U}-\widetilde{\ml{U}}\|_{\ml{X}_{T_*}}\left(\|\ml{U}\|_{\ml{X}_{T_*}}^{p-1}+\|\widetilde{\ml{U}}\|_{\ml{X}_{T_*}}^{p-1} \right).\label{A-priori}
\end{align}

Taking $k=0,1$, let us employ the $(L^2\cap L^1)-\dot{H}^k$ estimate \eqref{Help-01} in $[0,t/2]$ and the $L^2-\dot{H}^k$ estimate \eqref{Help-02} in $[t/2,t]$ to derive
\begin{align*}
\|\psi_{\nlin}(t,\cdot)-\widetilde{\psi}_{\nlin}(t,\cdot)\|_{\dot{H}^k}&=\left\|\int_0^tK_3^{\psi}(t-\eta,\partial_x)\big(|\psi^N(\eta,\cdot)|^p-|\widetilde{\psi}^N(\eta,\cdot)|^p\big)\,\mathrm{d}\eta\,\right\|_{\dot{H}^k}\\
&\lesssim \int_0^{t/2}(1+t-\eta)^{-\frac{1}{4}-\frac{k}{2}}\left\||\psi^N(\eta,\cdot)|^p-|\widetilde{\psi}^N(\eta,\cdot)|^p\right\|_{L^2\cap L^1}\mathrm{d}\eta\\
&\quad+\int_{t/2}^t(1+t-\eta)^{-\frac{k}{2}}\left\||\psi^N(\eta,\cdot)|^p-|\widetilde{\psi}^N(\eta,\cdot)|^p\right\|_{L^2}\mathrm{d}\eta.
\end{align*}
According to the a prior estimate \eqref{A-priori} with $r=1$ and $r=2$ one arrives at
\begin{align}
&(1+t)^{\frac{1}{4}+\frac{k}{2}}\|\psi_{\nlin}(t,\cdot)-\widetilde{\psi}_{\nlin}(t,\cdot)\|_{\dot{H}^k}\notag\\
&\lesssim(1+t)^{\frac{1}{4}+\frac{k}{2}}\int_0^{t/2}(1+t-\eta)^{-\frac{1}{4}-\frac{k}{2}}(1+\eta)^{-\frac{p}{2}+\frac{1}{2}}\,\mathrm{d}\eta\,\|\ml{U}-\widetilde{\ml{U}}\|_{\ml{X}_{T_*}}\left(\|\ml{U}\|_{\ml{X}_{T_*}}^{p-1}+\|\widetilde{\ml{U}}\|_{\ml{X}_{T_*}}^{p-1} \right)\notag\\
&\quad+(1+t)^{\frac{1}{4}+\frac{k}{2}}\int_{t/2}^t(1+t-\eta)^{-\frac{k}{2}}(1+\eta)^{-\frac{p}{2}+\frac{1}{4}}\,\mathrm{d}\eta\,\|\ml{U}-\widetilde{\ml{U}}\|_{\ml{X}_{T_*}}\left(\|\ml{U}\|_{\ml{X}_{T_*}}^{p-1}+\|\widetilde{\ml{U}}\|_{\ml{X}_{T_*}}^{p-1} \right)\notag\\
&\lesssim \left(\int_0^{t/2}(1+\eta)^{-\frac{p}{2}+\frac{1}{2}}\,\mathrm{d}\eta\, +(1+t)^{\frac{3}{2}-\frac{p}{2}}\right)\|\ml{U}-\widetilde{\ml{U}}\|_{\ml{X}_{T_*}}\left(\|\ml{U}\|_{\ml{X}_{T_*}}^{p-1}+\|\widetilde{\ml{U}}\|_{\ml{X}_{T_*}}^{p-1} \right)\label{Est-new-1}
\end{align}
with $p\geqslant 2$, where we used $1+t-\eta\approx 1+t$ for $\eta\in[0,t/2]$ and $1+\eta\approx 1+t$ for $\eta\in[t/2,t]$.  For the integral, since the condition $p>3$ one may obtain the uniform integrability over $[0,t/2]$. As a consequence,
\begin{align*}
(1+t)^{\frac{1}{4}+\frac{k}{2}}\|\psi_{\nlin}(t,\cdot)-\widetilde{\psi}_{\nlin}(t,\cdot)\|_{\dot{H}^k}\lesssim \|\ml{U}-\widetilde{\ml{U}}\|_{\ml{X}_{T_*}}\left(\|\ml{U}\|_{\ml{X}_{T_*}}^{p-1}+\|\widetilde{\ml{U}}\|_{\ml{X}_{T_*}}^{p-1} \right)
\end{align*}
with $k=0,1$, which implies the desired estimate \eqref{Crucial-02} holding for the second element of $\ml{U}_{\nlin}$ and, in turn, $\ml{U}$ immediately.

For another, by employing the derived $(L^2\cap L^1)-L^2$ estimate \eqref{Help-03} in $[0,t]$ one notices
\begin{align*}
&(1+t)^{-\frac{1}{4}}\|w_{\nlin}(t,\cdot)-\widetilde{w}_{\nlin}(t,\cdot)\|_{L^2}\\
&\lesssim(1+t)^{-\frac{1}{4}}\int_0^t(1+t-\eta)^{\frac{1}{4}}\left\||\psi^N(\eta,\cdot)|^p-|\widetilde{\psi}^N(\eta,\cdot)|^p\right\|_{L^2\cap L^1}\mathrm{d}\eta\\
&\lesssim \left(\int_0^{t/2}(1+\eta)^{-\frac{p}{2}+\frac{1}{2}}\,\mathrm{d}\eta+(1+t)^{\frac{1}{4}-\frac{p}{2}}\int_{t/2}^t(1+t-\eta)^{\frac{1}{4}}\,\mathrm{d}\eta\right)\|\ml{U}-\widetilde{\ml{U}}\|_{\ml{X}_{T_*}}\left(\|\ml{U}\|_{\ml{X}_{T_*}}^{p-1}+\|\widetilde{\ml{U}}\|_{\ml{X}_{T_*}}^{p-1} \right)\\
&\lesssim \|\ml{U}-\widetilde{\ml{U}}\|_{\ml{X}_{T_*}}\left(\|\ml{U}\|_{\ml{X}_{T_*}}^{p-1}+\|\widetilde{\ml{U}}\|_{\ml{X}_{T_*}}^{p-1} \right)
\end{align*}
thanks to $p>3$. With the same approach of \eqref{Est-new-1}, i.e. the $(L^2\cap L^1)-\dot{H}^2$ estimate \eqref{Help-03} in $[0,t/2]$ and  the $L^2-\dot{H}^2$ estimate \eqref{Help-04} in $[t/2,t]$, it concludes
\begin{align*}
(1+t)^{\frac{3}{4}}\|w_{\nlin}(t,\cdot)-\widetilde{w}_{\nlin}(t,\cdot)\|_{\dot{H}^2}\lesssim \|\ml{U}-\widetilde{\ml{U}}\|_{\ml{X}_{T_*}}\left(\|\ml{U}\|_{\ml{X}_{T_*}}^{p-1}+\|\widetilde{\ml{U}}\|_{\ml{X}_{T_*}}^{p-1} \right).
\end{align*}

These estimates imply our desired estimates \eqref{Crucial-02}, and \eqref{Crucial-01} if we apply \eqref{Crucial-03}, immediately. Associated with the Banach fixed point argument our proof is finished.

\section*{Acknowledgments} 
 Wenhui Chen is supported in part by the National Natural Science Foundation of China (grant No. 12301270, grant No. 12171317), 2024 Basic and Applied Basic Research Topic--Young Doctor Set Sail Project (grant No. 2024A04J0016), Guangdong Basic and Applied Basic Research Foundation (grant No. 2023A1515012044).

\end{document}